\documentclass[final,1p]{elsarticle}

\usepackage[utf8]{inputenc}
\usepackage[dvipsnames]{xcolor}
\usepackage{amsmath,amssymb,amsthm}
\usepackage{physics}
\usepackage{hyperref}
\usepackage{subcaption}
\usepackage{bbm}
\usepackage{makecell}

\usepackage{graphicx}%
\usepackage{multirow}%
\usepackage{amsmath,amssymb,amsfonts}%
\usepackage{amsthm}%
\usepackage{mathrsfs}%
\usepackage[title]{appendix}%
\usepackage[dvipsnames]{xcolor}%
\usepackage{textcomp}%
\usepackage{manyfoot}%
\usepackage{booktabs}%
\usepackage{algorithm}%
\usepackage{algorithmicx}%
\usepackage{algpseudocode}%
\usepackage{listings}%
\usepackage{pdflscape}

\usepackage[T1]{fontenc}
\usepackage[english]{babel}
\usepackage{tikz}
\usepackage{pgfplots}
\usetikzlibrary{patterns}
\pgfplotsset{
    tick label style={
        /pgf/number format/use comma=false,
        /pgf/number format/1000 sep={}
    }
}

\usepackage{xcolor}

\hypersetup{colorlinks = true, allcolors = blue}
\usepackage[nameinlink,noabbrev,capitalize]{cleveref}

\newcommand{\Real}{\mathbb{R}}

\newcommand{\paren}[1]{\left( #1 \right)}

\newtheorem{theorem}{Theorem}[section]

\newtheorem{corollary}{Corollary}[section]
\newtheorem{lemma}{Lemma}[section]
\newtheorem{definition}{Definition}[section]
\newtheorem{hypothesis}{Hypothesis}
\theoremstyle{remark}
\newtheorem*{remark}{Remark}

\input Typocaps.fd

\definecolor{customOrange}{HTML}{EB8934}
\definecolor{customTeal}{HTML}{538A96}
\definecolor{customBlue}{HTML}{34C9EB}
\definecolor{customBrown}{HTML}{6B5B4D}

\begin{document}

\begin{frontmatter}
% \title{Accelerating High-Fidelity Fixed Point Schemes with On-the-fly Reduced Order modelling}
\title{General Framework and Error Estimates for ROM-accelerated Fixed Point Iterations}

\author[GIREF]{Philippe-André Luneau}
\ead{philippe-andre.luneau.1@ulaval.ca}
\author[GIREF]{Jean Deteix}
\address[GIREF]{Groupe Interdisciplinaire de Recherche en \'El\'ements Finis de l'Universit\'e Laval, D\'epartement de Math\'ematiques et Statistique, Universit\'e Laval, Qu\'ebec, Canada}
\cortext[cor1]{Corresponding author}

\begin{abstract}
% % FIXED: "need to do" -> "need for"; "high fidelity" -> "high-fidelity" (compound adjective)
% A general framework for accelerating fixed point schemes for problems related to partial differential equations is presented in this article. The speedup is obtained by training a reduced-order model on-the-fly, removing the need for an offline training phase and any dependence on a precomputed reduced basis (e.g.\ a fixed geometry or mesh). The surrogate model can adapt itself along the iterations because of a novel error criterion based on error propagation, ensuring the high-fidelity character of the converged result. Convergence results are given for a general class of fixed point problems with complex dependence structures between multiple auxiliary linear systems. The proposed algorithm is applied to the solution of a system of coupled partial differential equations. The speedups obtained are significant, and the output of the method can be considered high-fidelity when compared to the reference solution.

Whether it is for solving nonlinear equations, optimization problems, or autonomous dynamical systems, fixed-point-type iterations are widely used in numerical sciences. 
On-the-fly reduced-order modelling (ROM) enables the construction of a low-dimensional, self-correcting approximation of the solution to this system during the iterative process, while removing the need to do an offline training phase and any dependence on a precomputed reduced basis (e.g., a fixed geometry or mesh). This technique has been used in specific fields before, including fluid-structure interactions and topology optimization, but no general study of this method has been done to the knowledge of the authors. A general method for accelerating fixed point schemes will be presented. We show that when the iteration mapping is contractive, the error of the approximate solution is guaranteed to be within the user-defined tolerance using inexact fixed-point theory. This methodology is then applied to the solution of systems of PDEs with a block Gauss-Seidel scheme. Errors due to the ROM are propagated through each iteration with respect to the computational graph of the system, which allows one to estimate whether the current iteration is still within the user-defined tolerance. Some working hypotheses necessary to observe a significant speedup and the limitations of the method are explored as well.  As a numerical illustration, the methodology is applied to a multiphysics lid-driven cavity flow in two dimensions and a multiphysics problem with an industrial geometry in three dimensions.

\end{abstract}


\end{frontmatter}

\section{Introduction}\label{sec1}

Acceleration of iterative schemes by reusing previously computed data is not a new idea, and has been studied during the last decades in multiple contexts. It is still a subject with a lot of potential in modern research~\citep{Saad_2025}. This article formalizes the already-known methodology to accelerate fixed-point schemes arising from the determination of solutions to partial differential equations (PDEs) using reduced-order models (ROMs) in the framework of inexact fixed point iterations~\citep{Alfeld_1982,Birken_2015}, and we propose a way to rigorously propagate the error throughout the iterations.

In optimization, especially in structural optimization, the concept of approximate reanalysis, consisting of using previous designs to reduce the computational cost of solving the state equation has been introduced by Kirsch in the eighties~\citep{Kirsch_1982}. This led to the development of the method of combined approximations~\citep{Kirsch_1999}, which has been shown to be equivalent to a particular case of the preconditioned conjugate gradient method~\citep{Kirsch_Kocvara_Zowe_2002}. These methods were generalized to topology optimization (TO) problems by Amir and Sigmund~\citep{Amir_Bendsøe_Sigmund_2009,Amir_Stolpe_Sigmund_2010,Amir_Sigmund_Lazarov_Schevenels_2012}. More recently, data-driven techniques have been used to learn a ROM on-the-fly during the optimization procedure, e.g.\ orthogonalization of previous iterates with Gram-Schmidt~\citep{Gogu_2015} and hyperreduction techniques like proper orthogonal decomposition (POD) for steady~\citep{Xiao_Lu_Breitkopf_Raghavan_Dutta_Zhang_2020,Xiao_Jun_Lu_Raghavan_Zhang_2022} and unsteady~\citep{Qian_2022} TO problems. Theoretical framework and proof of convergence have been established for on-the-fly ROM combined with trust-region methods~\citep{Yano_Huang_Zahr_2021,Wen_Zahr_2023} and line search methods~\citep{Grundvig_Heinkenschloss_2024}.

In the field of fluid-structure interaction (FSI), iterative or fixed point-like strategies are popular to solve the coupling between the solid and fluid equations, leading to large computational cost. Offline ROMs have been used recently in combination with quadratic POD approximations~\citep{Tiba_Dairay_De_Vuyst_Mortazavi_Berro_Ramirez_2024}, but an on-the-fly ROM framework in combination with fixed point and quasi-Newton methods has been used in the literature~\citep{Vierendeels_Lanoye_Degroote_Verdonck_2007,Degroote_Bathe_Vierendeels_2009,Delaissé_Demeester_Fauconnier_Degroote_2022}.

Algebraic solvers have also benefited from advances in reduced-order modelling, for instance linear solvers have used Krylov subspace recycling for a long time~\citep{Soodhalter_de_Sturler_Kilmer_2020}, and more recently combined it with POD-based approaches~\citep{Carlberg_Forstall_Tuminaro_2016,Liu_Wang_Yu_Wang_2021} to augment the Krylov subspace in meaningful ways. \cite{Elman_Su_2020} uses low-rank approximation to accelerate the GMRES method for the solution of Navier-Stokes equation in an all-at-once space-time paradigm. Nonlinear solvers for complex unsteady nonlinear mechanics problems have leveraged ROM-based subspace augmentation as well~\citep{Kerfriden_Gosselet_Adhikari_Bordas_2011}. Other works have studied the acceleration of Picard iterations, for instance by using partial solving~\citep{Senecal_Ji_2017} or relaxation schemes~\citep{Birken_Gleim_Kuhl_Meister_2015}.

In this work, a general framework for on-the-fly adaptive ROM construction for fixed point methods will be presented. The main contributions are:
\begin{itemize}
    \item A generic algorithm and the formalism required to accelerate fixed point iterations using an inexact fixed-point operator;
    \item Convergence results for a fixed-point scheme dependent on the solution of multiple coupled linear systems;
    \item An error estimator based on error propagation to quantify the error induced by a ROM throughout the fixed point iterations.
\end{itemize}

As shown in the previous paragraph, on-the-fly ROMs have been used extensively in the literature. However, even if they are a very practical tool, a formal analysis of their behaviour has yet to be established. To the knowledge of the authors, such an approach at error estimation has not been used in the literature in this context; a simple a posteriori proxy for the error is often used instead \cite{Gogu_2015,Xiao_Lu_Breitkopf_Raghavan_Dutta_Zhang_2020,Xiao_Jun_Lu_Raghavan_Zhang_2022,Qian_2022,Kerfriden_Gosselet_Adhikari_Bordas_2011,Wen_Zahr_2023} to quantify the precision of the ROM (e.g.\ the relative residual), which does not take into account the error propagation through the iterative process and can lead to premature convergence if not treated properly. An example of this will be shown in~\Cref{section:ROMQM}.
% We propose general error estimates that can be computed on-the-fly to measure the real error induced by the ROM in the iterative process.
Moreover, since many iterative processes can be reframed in the language of fixed point iterations (e.g.,  optimization, FSI, autonomous dynamical systems), the global methodology presented could be adapted to a large variety of problems in applied mathematics.

First, general theory about inexact fixed point iterations and reduced-order modelling will be presented (\Cref{section:preliminaries}). Then, the main algorithm and its theoretical guarantees will be given (\Cref{section:methodo}), when applied to the solution of an arbitrary number of coupled linear systems. Multiple convergence criteria will be given for such large coupled nonlinear systems. Finally, the algorithm will be applied to a problem in thermo-fluidics (\Cref{sec:application}), where coupled PDEs are solved in an alternate manner, which can be formulated as a fixed point scheme.

\section{Preliminaries}\label{section:preliminaries}

A quick overview of the theory of fixed point approximations with exact and inexact iterative methods will be presented, followed by an introduction to fundamental concepts of reduced-order modelling using POD.

\subsection{Fixed Point Iterations}

Let $G:\Real^n \to \Real^n$ be a selfmap. The definition of Lipschitz continuity is recalled.
\begin{definition}
    $G$ is $L$-Lipschitz continuous if for any $\mathbf{x},\mathbf{z}\in \Real^n$, there is a constant $L>0$ such that
    \begin{equation*}
        \| G(\mathbf{x}) - G(\mathbf{z})\| \leq L \| \mathbf{x} - \mathbf{z} \|.
    \end{equation*}
\end{definition}
In the case where $L<1$, $G$ is said to be a \textit{contraction}. The well-known theorem from~\cite{Banach_1922} (\S~2, Thm.~6) ensures that a contractive mapping has a unique fixed point and that the sequence generated by the so-called \textit{Picard iterates}, starting from a point $\mathbf{x}^0$,
\begin{align*}
    \mathbf{x}^{k+1} = G(\mathbf{x}^k),
\end{align*}
converges to the fixed point. In the case where $L\geq 1$, a function might still admit a unique fixed point but some other requirement has to be satisfied, for instance, strong pseudocontractivity~\citep{Chidume_1994,Inchan_2009,Jorquera_Alvarez_2023}. Iterative schemes can be designed to approximate fixed points in these situations, like the \textit{Krasnoselskij iterates}
\begin{equation}
\label{krasno}
    \mathbf{x}^{k+1} = G_\lambda(\mathbf{x}^k) = (1-\lambda) \mathbf{x}^k + \lambda G(\mathbf{x}^k)
\end{equation}
or the \textit{Mann iterates}
\begin{equation*}
\label{mann}
    \mathbf{x}^{k+1} = G_{\lambda^k}(\mathbf{x}^k) = (1-\lambda^k) \mathbf{x}^k + \lambda^k G(\mathbf{x}^k),
\end{equation*}
which both converge to the unique fixed point under conditions relating the pseudocontraction constants and the relaxation factor $\lambda$ (resp.\ $\lambda^k$)~\citep{Berinde_2007}. Under the right hypotheses, the averaged operators $G_\lambda$ and $G_{\lambda^k}$ are contractions as well. Those iterative schemes arise especially in the context of operator splitting~\citep{Liang_Fadili_Peyré_2016}. One could also obtain a constant step ($\gamma$) gradient descent scheme from~\eqref{krasno} by taking $G(\mathbf{x}) = \mathbf{x} - \gamma \nabla J (\mathbf{x})$, which indeed satisfies the convergence hypotheses in the case of $J$ strongly convex and $\nabla J$ Lipschitz continuous~\citep{Jung_2017,Wensing_Slotine_2020}. Those techniques are mentioned here because the error propagation analysis performed in the next section could be directly applied to those iterative schemes instead of the Picard scheme, thus making the methodology presented in the following section applicable to a wide range of problems.

In the case of nonlinear problems where the evaluation of $G$ is itself obtained by an iterative process, or in the case where a data-based surrogate model is used for $G$, the evaluation at a point $G(\mathbf{x})$ may be inexact. Let $G_k:\mathbb{R}^n\to\mathbb{R}^n$ be an approximation of $G$ used at the $k$-th iteration of the iterative process. The inexact Picard iterations $\widehat{\mathbf{x}}^k$ are defined as
\begin{align*}
    \widehat{\mathbf{x}}^{k+1} = G_k(\widehat{\mathbf{x}}^k),
\end{align*}
starting from an initial datum $\mathbf{x}^0$.

\begin{hypothesis}
    \label{hyp:controled}
    At the $k$-th iteration of the inexact fixed point algorithm, the approximation error satisfies
    \begin{equation*}
        \| G(\widehat{\mathbf{x}}^k) - G_k(\widehat{\mathbf{x}}^k) \|\leq \delta^k.
    \end{equation*}
    for some $\delta^k > 0$ depending on $\widehat{\mathbf{x}}^k$.
\end{hypothesis}

Early work by Alfeld studies the convergence of inexact fixed point algorithms. The main result from~\cite[\S~2, Thm.~3]{Alfeld_1982} is that a necessary and sufficient condition for the convergence of the inexact fixed point scheme to the fixed point of $G$ is that $\delta^k \to 0$ as $k\to \infty$. \cite[\S~2, Thm.~4]{Birken_2015} specialized this result to alternate solution of systems (like FSI), which will be used later not as a stopping criterion as it was intended, but as a ROM quality indicator (\S\ref{sec:application}). The result is the following.
\begin{theorem}
\label{thm:birken}
    Let $G$, $H$ be $\mathbb{R}^n$-selfmaps with respective Lipschitz constants $L_1$, $L_2$ s.t.\ $L_1L_2<1$. Assuming the sequences $\mathbf{x}^{k+1} = H(G(\mathbf{x}^k))$ and $\widehat{\mathbf{x}}^{k+1} = H_k(G_k(\widehat{\mathbf{x}}^k))$ both converge to $\mathbf{x}^*$ and $\widehat{\mathbf{x}}^*$ respectively, and~\Cref{hyp:controled} is satisfied for sequences $\delta_1^k \leq \delta_1$ and $\delta_2^k \leq \delta_2$. Then
    \begin{align*}
        \| \widehat{\mathbf{x}}^* - \mathbf{x}^* \| \leq \frac{L_2 \delta_1 + \delta_2}{1 - L_1 L_2}.
    \end{align*}
\end{theorem}
In particular, when $H_k = H$ for all $k$ ($\delta_2=0$),
\begin{align*}
    \| \widehat{\mathbf{x}}^* - \mathbf{x}^* \| \leq\frac{L_2}{1 - L_1 L_2} \delta_1.
\end{align*}
If by construction of $G_k$, it is possible to establish an upper bound on $\delta_1^k$, then it can be ensured at each iteration that
\begin{align}
    \delta_1^k \leq \delta_1 \leq \frac{1 - L_1 L_2}{L_2}\varepsilon,
    \label{eq:asymptote}
\end{align}
for $\varepsilon$ the tolerance of the algorithm. We can hope that asymptotically, the error will satisfy the bound when the stopping criterion is attained (e.g.\ $\|\widehat{\mathbf{x}}^{k+1}-\widehat{\mathbf{x}}^k\|\leq \varepsilon$). In the next section, a strategy to design a non-asymptotic criterion based on error propagation in the inexact sequence will be proposed, based on the analyses of Alfeld and Birken.

Recent literature has studied similar results for inexact Mann-like iterative schemes \citep{Liu_1995,Xu_1998,Liang_Fadili_Peyré_2016,Bravo_Cominetti_Pavez-Signé_2019} and inexact Newton's method~\citep{Argyros_1999}.

\subsection{Projection-based ROMs and Full-field Approximations}
\label{sec:roms}

Consider the problem
\begin{align*}
    \Lambda(\mu) u = f(\mu)
\end{align*}
for $\Lambda$ some linear partial differential operator depending on a parameter $\mu \in \mathcal{P}$, with suitable boundary conditions to ensure existence and uniqueness.

By using standard discretization techniques like the finite element method (FEM)~\citep{Ern_Guermond_2021}, if $\Lambda$ is elliptic, this problem can be approximated by the solution of a finite-dimensional linear system
\begin{align}
\label{eq:discreteprob}
    A(\mu)u_h = F(\mu)
\end{align}
with a symmetric positive definite matrix $A(\mu)$ of dimension $N$.

The solution $u_h$ is assumed to be continuously dependent on data $\mu$ (requirements to satisfy this hypothesis can be found in~\cite{Hinze_Pinnau_Ulbrich_Ulbrich_2009}). Intuitively, this means that solutions associated to parameters sufficiently close to each other will behave similarly. Leaning into this paradigm, a data-driven approach to approximate solutions based on precomputed snapshots (solutions at other values of $\mu$) can be considered.

The goal of projection-based ROMs (PROMs) is to find a linear subspace $\mathcal{V}$ on which to project~\eqref{eq:discreteprob} while minimizing some notion of projection error. This linear subspace should be spanned by a basis of size much smaller than the ambient space to reduce the cost as much as possible. A maximal size for the reduced basis $N_b$ is first chosen, i.e.\ $\mathrm{dim}\mathcal{V}\leq N_b$. A snapshot matrix $S$ of size $N\times N_b$ is built. Every column is the solution $u_h^{(i)}$ to~\eqref{eq:discreteprob} for some parameter value $\mu^{(i)}$ ($i=1,\dots,N_b$). Ideally, the subspace $\mathcal{V}$ should be taken as the one that minimizes the high-dimensional embedding error of the projected snapshots,
\begin{equation*}
    \min_{\substack{ \mathrm{dim}\mathcal{V}=P \\ P\leq N_b}}\| \mathcal{E}_\mathcal{V} ( \Pi_\mathcal{V} (u_h^{(i)})) - u_h^{(i)}\|,
\end{equation*}
where $\Pi_\mathcal{V}:\Real^N \to \mathcal{V}$ is the low-dimensional projection and $\mathcal{E}_\mathcal{V}:\mathcal{V} \to \Real^{N}$ is the high-dimensional embedding.

Two common basis reduction methods are based on the Gram-Schmidt (GS) process and singular value decomposition (SVD). First, the snapshot matrix is centered by subtracting the mean snapshot to each column, $\overline{u} = \frac{1}{N_b} \sum_{i=1}^{N_b} u_h^{(i)}$. This is equivalent to right-multiplying $S$ by the centering matrix $C_{N_b} = I_{N_b} - \frac{1}{N_b}J_{N_b}$. The projection matrix $V$ can be obtained either by the Gram-Schmidt algorithm (orthogonalizing the columns of $S C_{N_b}$) or by singular value decomposition (SVD),
\begin{align*}
    SC_{N_b} = U \Sigma W^\top.
\end{align*}

In the case of simple orthogonalization, the basis is always of size $N_b$. In the case of the SVD, an energy criterion~\citep{Quarteroni_Manzoni_Negri_2016} can be used to determine the fraction of the variance captured by a basis of size $P<N_b$ composed of the $P$ first columns of $U$,
\begin{align*}
    \frac{\sum_{i=1}^P \sigma_i^2}{\sum_{i=1}^{N_b} \sigma_i^2}< 1-\varepsilon_\mathrm{rb}^2,
\end{align*}
where $\sigma_i$ are the diagonal elements of $\Sigma$, or the singular values of $SC_{N_b}$. Those $P$ columns will be grouped in a matrix $V \in \Real^{N\times P}$, and $\Pi_\mathcal{V} = V^\top$ and $\mathcal{E}_\mathcal{V}=V$. To find an approximation to $u_h$ for a new value of $\mu$, the following low-order projected linear system has to be solved
\begin{align}
    \label{eq:lowdimsys}
    V^\top A(\mu) V\alpha = V^\top F(\mu) - V^\top K(\mu) \overline{u},
\end{align}
where $\alpha \in \Real^{P}$, and the POD approximation to $u_h$ is given by $\widehat{u} = V \alpha + \overline{u}$.

The error in Euclidean norm is upper bounded by
\begin{align}
\label{eq:errgeneralROM}
    \norm{u_h - \widehat{u}} \leq \norm{A(\mu)^{-1}}\norm{r(\mu,\widehat{u})},
\end{align}
where $r(\mu,\widehat{u}) = F(\mu) - A(\mu)\widehat{u}$ is the residual.

\section{Methodology}\label{section:methodo}

First, an adaptive error estimator for inexact fixed point iterations will be derived. This estimator can be computed on-the-fly. Then, the setting of coupled linear systems solved through Picard iterations will be presented, and the quantities needed to compute the aforementioned error estimator will be derived in this context.

\subsection{Error Propagation}
\label{section:errorprop}

The main idea in this work will be to ensure that the distance between a fixed point sequence and an inexact fixed point sequence using a ROM approximation stays under a desired tolerance. Incidentally, picking this tolerance to be at most equal to the tolerance $\varepsilon$ of the fixed point algorithm (in the following, this tolerance will be set to $\varepsilon$ as well) will yield a solution very close to the one that would be obtained by using only high-fidelity data. First, results from~\cite{Alfeld_1982} and~\cite{Birken_2015} are adapted to measure the distance between the exact and inexact sequence, starting from a common point, after an arbitrary number of iterations.

\begin{lemma}\label{GrbErrBoundK}
Assuming $G$ is $L$-Lipschitz continuous, starting from a common point $\mathbf{x}^j=\widehat{\mathbf{x}}^j$, the distance between the $k+1$-th inexact iterate and the $k+1$-th exact iterate ($k\geq j$) is upper bounded by
\begin{equation}
        \| \widehat{\mathbf{x}}^{k+1} - \mathbf{x}^{k+1}\| \leq  \eta^{k+1} := \sum_{i=0}^{k-j} L^i \delta^{k-i}.
        \label{eq:etak}
\end{equation}
\end{lemma}

\begin{proof} The error is split into two parts,
\begin{equation*}
    \| \widehat{\mathbf{x}}^{k+1} - \mathbf{x}^{k+1}\|  \leq \| G(\widehat{\mathbf{x}}^{k}) - G(\mathbf{x}^{k}) \| + \| G_{k}(\widehat{\mathbf{x}}^{k}) - G(\widehat{\mathbf{x}}^{k}) \|.
\end{equation*}
The first term is related to the sensitivity of $G$ w.r.t.\ the error on the previous point. By Lipschitz continuity,
\begin{equation*}
    \| G(\widehat{\mathbf{x}}^{k}) - G(\mathbf{x}^{k}) \| \leq L \| \widehat{\mathbf{x}}^{k} - \mathbf{x}^{k} \|.
\end{equation*}
The second term can be bounded by~\Cref{hyp:controled}:
\begin{equation*}
    \| G_{k}(\widehat{\mathbf{x}}^{k}) - G(\widehat{\mathbf{x}}^{k}) \| \leq \delta^{k}.
\end{equation*}
The error bound depends on the error at the previous step:
\begin{equation*}
    \| \widehat{\mathbf{x}}^{k+1} - \mathbf{x}^{k+1}\| \leq L \| \widehat{\mathbf{x}}^{k} - \mathbf{x}^{k} \| + \delta^{k}.
\end{equation*}
By applying the splitting above recursively $k-j$ times, until reaching $\| \widehat{\mathbf{x}}^{j} - \mathbf{x}^{j} \|=0$, the desired upper bound is obtained.
\end{proof}

Depending on the choice of inexact fixed point operator $G_k$ (its nature left intentionally ambiguous for now), the quantity $\delta^k$ could be evaluated numerically along the iterative process and the approximate propagated error $\eta^k$ could be quantified explicitly. The Lipschitz constant $L$ can be easily approximated numerically by considering the ratio of distances between successive iterations,
\begin{equation*}
    L \approx
    \frac{\|\mathbf{x}^{k+1}-\mathbf{x}^k\|}{\|\mathbf{x}^{k}-\mathbf{x}^{k-1}\|}.
\end{equation*}

The next step is not only to measure the propagated error, but to control it so that it remains below the prescribed tolerance at all times. If the mapping $G$ is non-expansive, it can be easily shown that using the exact fixed point operator instead of the inexact fixed point operator will ensure the error between the next iterates of the inexact and exact sequences is below the threshold.

\begin{lemma}
\label{GrbCtrlErr}
Assuming $G$ is nonexpansive ($L\leq 1$), if  $\| \widehat{\mathbf{x}}^k-\mathbf{x}^k\| \leq \varepsilon$, then
\begin{align*}
        \| G(\widehat{\mathbf{x}}^k) - \mathbf{x}^{k+1}\| \leq \varepsilon.
\end{align*}
\end{lemma}
\begin{proof}
\begin{equation*}
    \| G(\widehat{\mathbf{x}}^k) - \mathbf{x}^{k+1}\| =  \| G(\widehat{\mathbf{x}}^k) - G(\mathbf{x}^{k})\| \leq L \| \widehat{\mathbf{x}}^k-\mathbf{x}^k\| \leq  L \varepsilon,
\end{equation*}
and $L\leq 1$.
\end{proof}

With these two properties in hand, an algorithm can be designed to ensure that the error will stay below the threshold. First, it will be assumed that a certain number of iterations (samples), $N_b$, have to be collected first to build the inexact fixed point operators. From there, the algorithm can be outlined:

\begin{enumerate}
    \item Until the maximal sample size $N_b$ is attained or the inexact operator error $\delta^k$ is under the tolerance, the exact fixed point is used.
    \item As long as the approximate propagation error $\eta^k$ is under the tolerance, the inexact fixed point is used.
    \item When $\eta^k$ gets above the threshold, the current step is rejected, and the previous (valid) iterate is used with the exact fixed point operator. The propagation error is set to $L$ times the previous error. The inexact fixed point operator is updated using the high-fidelity data from this data point.
    \item Steps 2 and 3 are repeated until convergence.
\end{enumerate}
A last validation step can be executed at the end of the algorithm to ensure that the convergence is true as well in terms of the exact fixed point. If it is not the case, a looping strategy can be applied to reach true convergence. The pseudocode of the algorithm is available at~\Cref{algo:acceleratedscheme}. From~\Cref{GrbErrBoundK} and~\Cref{GrbCtrlErr}, a theoretical guarantee can be obtained with regards to the sequence generated by the algorithm.

\begin{algorithm}[hpb]
\caption{On-the-fly Accelerated Inexact Picard Iterations}\label{algo:acceleratedscheme}%
\begin{algorithmic}[1]
\State $k=0$
\State $\eta^k=\infty$
\State $recompute=False$
\State $converged=False$
\While{$k<k_{max}$ \textbf{and} \textbf{not} $converged$}

    \State $step=\infty$
    \If{$\eta^k>\varepsilon$ \textbf{or} $k<N_b$ \textbf{or} $recompute$}

        \State $\mathbf{x}^{k+1} = G(\mathbf{x}^k)$
        \State Construct $G_k$, $\delta^k$ with $\mathbf{x}^{k+1}$
        \If{$recompute$}
            \State $\eta^{k+1} = L\times \eta^{k}$ \Comment{Error after refinement}
            \State $recompute = False$
        \Else
            \State $\eta^{k+1} = \delta^{k}$\;
        \EndIf

    \Else
        \State $G_k = G_{k-1}$
        \State $\mathbf{x}_t = G_k(\mathbf{x}^k)$
        \State $newerr = \delta^k + L\times \eta^k$
        \If{$newerr>\varepsilon$} \Comment{Reject inexact step}
            \State $\mathbf{x}^{k+1} = \mathbf{x}^k$
            \State $recompute=True$

        \Else \Comment{Accept inexact step}
            \State $\mathbf{x}^{k+1} = \mathbf{x}_t$
            \State $\eta^{k+1} = newerr$

        \EndIf
    \EndIf
    \State $step = \norm{\mathbf{x}^{k+1}-\mathbf{x}^k}$\;
    \If{$step < \varepsilon$ \textbf{and} \textbf{not} $recompute$}
        \If{$\norm{G(\mathbf{x}^{k+1})-\mathbf{x}^{k+1}}< \varepsilon$} \Comment{Validate convergence}
            \State $converged = True$
        \Else
            \State $recompute = True$
        \EndIf
    \EndIf
    \State $k\leftarrow k+1$
 \EndWhile
\end{algorithmic}
\end{algorithm}

\begin{corollary}
    Assuming $G$ is a contraction, the sequence generated by~\Cref{algo:acceleratedscheme} will stay at a distance at most $\varepsilon$ from the sequence generated by the exact fixed point scheme.
    \label{cor:main}
\end{corollary}

It is important to notice that the algorithm guarantees precision, but not acceleration. If the costs of computing and evaluating $G_k$ and $\delta^k$ are low enough with respect to computing $G$, the algorithm will give a speedup. A final hypothesis is that it is possible to construct a $G_k$ such that $\delta^k$ will be below the given tolerance $\varepsilon$. In the case where that is not possible, one should still expect the algorithm to work because $G$ will always be used instead of $G_k$, but at a higher computational cost since $\delta^k$ will be unnecessarily computed at every iteration.

\begin{remark}
\label{remark:certified_vs_practical}
Two regimes of \Cref{algo:acceleratedscheme} can be distinguished.
In the \textit{certified regime}, $\delta^k$ satisfying 
\Cref{hyp:controled} is replaced by a provably valid upper bound (e.g., derived from coercivity constants and Poincaré 
inequalities) and the 
guarantee of \Cref{cor:main} holds strictly. In the 
\textit{practical regime}, the constants are estimated numerically 
from the iterates (see \Cref{section:alternatesolution}); 
the bound of \Cref{cor:main} then holds approximately. The quality 
of this approximation is measured through the effectivity of 
$\eta^k$ with respect to the true error, 
which will be studied in \Cref{section:ROMQM}.
\end{remark}

% \subsection{Continuation on the tolerance}

% As mentioned previously, it is not always possible to construct an approximate operator $G_k$ such that the error will be under the tolerance $\varepsilon$ (this is problem dependant). The algorithm above can be relaxed to accept ROM steps with an error greater than $\varepsilon$. First, set another tolerance $\varepsilon_\mathrm{ROM}$, larger than $\varepsilon$, that will be used to determine if $\eta^k$ is small enough and accept the ROM step $\widehat{\mathbf{x}}^k$. This will allow to use the ROM more frequently, even if the error is above the required final tolerance. However, if the fixed point stepsize $\Delta \mathbf{x}^k$ is smaller than $\eta^k$, the step is not meaningful, in the sense that it is similar to or smaller than the error in the ROM. In that case, the auxiliary tolerance should be decreased by some rule, e.g.,
% \begin{align*}
%     \varepsilon_\mathrm{ROM} \leftarrow \max(\gamma\varepsilon_\mathrm{ROM}, \varepsilon),
% \end{align*}
% for some user-defined $\gamma\in (0,1)$. This continuation will progressively tighten the tolerance until it reaches the high-fidelity tolerance. Adaptive tolerance is also present in the work \cite{Xiao_Jun_Lu_Raghavan_Zhang_2022} in an optimization context.

\subsection{ROM-based Inexact Operator}

The algorithm presented in the previous subsection is very general. However, the explicit construction of the inexact operators $G_k$ must be detailed. Most of the time, the costly part of the evaluation of the operator $G$ will be related to the solution of some PDE. Depending on what is the problem at hand, $G_k$ might be obtained in multiple ways (e.g.\ a shallow or physics-informed neural network, a response surface, or any other surrogate requiring few training data and with fast training time). In this case a ROM using POD is considered, and the previous iterations will be used as snapshots, as in~\cite{Vierendeels_Lanoye_Degroote_Verdonck_2007,Xiao_Lu_Breitkopf_Raghavan_Dutta_Zhang_2020,Xiao_Jun_Lu_Raghavan_Zhang_2022}. The approximate solution operators are updated when the error in~\Cref{algo:acceleratedscheme} gets above the tolerance. The new snapshot is added to the snapshot matrix in a ``first in, first out'' (FIFO) pattern to keep $N_b$ snapshots in total. This can be seen as a non-parametric approach to the ROM methodology, because there is no parameter $\mu$ changing along the iterations, but since the fixed point iterations are convergent, the continuity argument mentioned in~\Cref{sec:roms} still holds.

\subsubsection{Picard Iterations Depending on the Solutions of Multiple Linear Systems}

The case where the fixed point operator $G$ depends on a variable $\mathbf{x}\in\Real^n$ but also on the solution of multiple linear systems, themselves dependent on $\mathbf{x}$, is studied. Let $A_i \in \mathrm{GL}_{n_i}(\Real)$, $F_i \in \Real^{n_i}$ ($i=1,\dots,p$) be the matrices and right-hand sides respectively of $p$ linear systems. It is assumed that those systems might have dependences on the solution of other linear systems $\mathbf{u}_i \in \Real^{n_i}$. In this context, it is assumed that there exists an order in which all the systems can be solved sequentially; otherwise, a nested nonlinear solver would be required, which will not be accounted for in the following analysis. The most general dependence scheme that can be considered is obtained by ordering the systems by the order in which they are solved, and making each system dependent on the solutions to every previous system, i.e.
\begin{align*}
    \mathbf{u}_i &:= A_i^{-1}F_i,\quad i=1,\dots,p,\\
    A_i &= A_i(\mathbf{x}, \mathbf{u}_{1},\dots,\mathbf{u}_{i-1}),\\
    F_i &= F_i(\mathbf{x}, \mathbf{u}_{1},\dots,\mathbf{u}_{i-1}).
\end{align*}
Since the couplings induce nonlinearity in this system of equations, the system needs to be solved in an iterative fashion. We will denote $\mu_i^{k+1} = (\mathbf{x}^k, \mathbf{u}^{k+1}_{1},\dots,\mathbf{u}^{k+1}_{i-1})$, the information passed from the previous iteration, $\mathbf{x}^k$, and the values $\mathbf{u}^{k+1}_{j}$ computed at the current iteration up to the $i$-th equation.

The dependences of $\mathbf{u}_i$ will not be written in the following to simplify the expressions, but are assumed to follow the rule established above, i.e.\ $\mathbf{u}_i^{k+1} = \mathbf{u}_i^{k+1}(\mu_i^{k+1})$. To represent that the computation of a single fixed-point iteration is dependent on the solution of those systems, a function
\begin{align*}
    \varphi : \Real^n \times \Real^{n_1} \times \dots \times \Real^{n_p} \to \Real^n
\end{align*}
is introduced, and the fixed point operator is defined as
\begin{equation*}
    G(\mathbf{x}) = \varphi(\mathbf{x}, \mathbf{u}_1,\dots,\mathbf{u}_p).
\end{equation*}
The function $\varphi$ can be seen as a multipoint iterative function without memory~\citep{Petković_Neta_Petković_Džunić_2014} with more complex dependence between the auxiliary quantities. Alternatively, this is also equivalent to a block Gauss-Seidel method~\citep{Saad_2003}. \cite{Whiteley_Gillow_Tavener_Walter_2011} have studied the error induced by neglecting couplings in a Gauss-Seidel scheme; in this work, the error made when using a ROM for an intermediate system is considered.

A few examples are frequently encountered when $p=1$. For instance, consider a nonlinear equation of the form
\begin{align*}
    A_1(\mathbf{x}) \mathbf{x} = F_1(\mathbf{x}),
\end{align*}
which could be solved by the following fixed point scheme
\begin{align*}
    A_1(\mathbf{x}^k) \mathbf{x}^{k+1} = F_1(\mathbf{x}^k).
\end{align*}
The corresponding Picard iterates are characterized by the function
\begin{align*}
    \varphi(\mathbf{x},\mathbf{u}_1) = \mathbf{u}_1.
\end{align*}
Another example can be obtained when thinking about constrained optimization,
\begin{align*}
    \min_{\mathbf{x}\in U} J(\mathbf{x},\mathbf{u}_1) \qq{s.t.} A_1(\mathbf{x})\mathbf{u}_1=F_1(\mathbf{x}).
\end{align*}
The fixed point scheme can be obtained by the projected gradient descent method, and the corresponding function is
\begin{align*}
\varphi(\mathbf{x},\mathbf{u}_1) = \Pi_U(\mathbf{x}-\gamma\nabla J(\mathbf{x},\mathbf{u}_1)),
\end{align*}
where $\Pi_U$ is the orthogonal projection onto $U$ and $\gamma$ is the step size. An example with $p=2$ will be presented in~\Cref{sec:application}. A sufficient condition for the convergence of such an iterative process is given.

\begin{theorem}
\label{thm:convergence}
    Assuming $\varphi$ is Lipschitz continuous in each of its arguments with Lipschitz constants $L_i$ ($i=0,\dots,p$), and that the solution maps $(\mathbf{x},\mathbf{u}_1,\dots,\mathbf{u}_{i-1})\mapsto \mathbf{u}_i$ are Lipschitz continuous with respect to each of their arguments, with constants $K_{i,j}$ ($i=1,\dots,p$, $j=0,\dots,i-1$). Assume
    \begin{align*}
       L_0 + \sum_{j=1}^p L_j \sum_{\sigma\in d_{0,j}} \prod_{m=1}^{\abs{\sigma}-1}K_{\sigma_m,\sigma_{m+1}} < 1,
    \end{align*}
    where $d_{i,j}$ is the set of strictly decreasing integer sequences starting at $j$ and ending at $i$. Then, $G$ is a contraction.
\end{theorem}
\begin{proof}
    An upper estimate to the Lipschitz constant of $G$ is computed. For an arbitrary pair of points $\mathbf{x},\mathbf{z}$, the triangle inequality gives
    \begin{align}
    \label{eq:lipphi}\abs{G(\mathbf{x})-G(\mathbf{z})} = \abs{\varphi(\mathbf{x},\mathbf{u}_1,\dots,\mathbf{u}_p) - \varphi(\mathbf{z},\mathbf{z}_1,\dots,\mathbf{z}_p)}\leq L_0 \|\mathbf{x} - \mathbf{z} \| +\sum_{j=1}^p L_j \| \mathbf{u}_j - \mathbf{z}_j\|
    \end{align}
    where $\mathbf{z}_i = \mathbf{u}_i(\mathbf{z},\mathbf{z}_1,\dots,\mathbf{z}_{i-1})$
    and,
    \begin{align}
        \label{eq:borneyj}
        \norm{\mathbf{u}_j-\mathbf{z}_j} \leq K_{j,0}\| \mathbf{x} - \mathbf{z}\| + \sum_{\ell = 1}^{j-1} K_{j,\ell} \norm{\mathbf{u}_\ell-\mathbf{z}_\ell}.
    \end{align}
    The statement of the theorem will be proven using strong induction. Assume that for $\ell=1,\dots,j-1$,
    \begin{align}
    \label{eq:RecurseHyp}
        \norm{\mathbf{u}_\ell-\mathbf{z}_\ell} \leq \sum_{\sigma\in d_{0,\ell}}\prod_{m=1}^{\abs{\sigma}-1}K_{\sigma_m,\sigma_{m+1}}\norm{\mathbf{x}-\mathbf{z}}.
    \end{align}
    Since all elements of $d_{0,\ell}$ start with $\sigma_1 = \ell$, when multiplying~\eqref{eq:RecurseHyp} by $K_{j,\ell}$, the product can be written as
    \begin{align*}
        \prod_{m=1}^{\abs{\zeta}-1}K_{\zeta_m,\zeta_{m+1}}\norm{\mathbf{x}-\mathbf{z}}
    \end{align*}
    for some $\zeta\in d_{0,j}$ with $\zeta_1=j$ and $\zeta_2=\ell$. By summing over $d_{0,\ell}$ as in~\eqref{eq:RecurseHyp}, all strictly decreasing integer sequences of the form $(j,\ell,\dots,0)$ are considered. Then, by summing over $\ell$ as in~\eqref{eq:borneyj}, all sequences of the form $(j,j-1,\dots,0)$, $(j,j-2,\dots,0)$ down to $(j,1,0)$ are considered. When adding the first term in~\eqref{eq:borneyj}, they are precisely all the sequences of $d_{0,j}$, meaning
    \begin{align}
    \label{eq:finalInduction}
        \norm{\mathbf{u}_j-\mathbf{z}_j} \leq \sum_{\sigma \in d_{0,j}} \prod_{m=1}^{\abs{\sigma}-1} K_{\sigma_m,\sigma_{m+1}}\norm{\mathbf{x}-\mathbf{z}}.
    \end{align}
    The base case is trivially satisfied, because if $j=1$,~\eqref{eq:borneyj} gives
    \begin{align*}
        \norm{\mathbf{u}_{1}-\mathbf{z}_{1}} \leq K_{1,0}\norm{\mathbf{x}-\mathbf{z}},
    \end{align*}
    and $(1,0)$ is the only strictly decreasing integer sequence starting at $1$ and ending at $0$. The case $j=2$ gives
    \begin{align*}
        \norm{\mathbf{u}_{2}-\mathbf{z}_{2}} &\leq K_{2,0}\norm{\mathbf{x}-\mathbf{z}} + K_{2,1}\norm{\mathbf{u}_{1}-\mathbf{z}_{1}}\\
        &\leq (K_{2,0}+K_{2,1}K_{1,0})\norm{\mathbf{x}-\mathbf{z}}
    \end{align*}
    and $(2,0)$ and $(2,1,0)$ are all the desired sequences.

     Combining~\eqref{eq:finalInduction} and~\eqref{eq:lipphi} gives an upper bound on the Lipschitz constant of $G$,
     \begin{align*}
          L \leq L_0 + \sum_{j=1}^p L_j \sum_{\sigma\in d_{0,j}} \prod_{m=1}^{\abs{\sigma}-1}K_{\sigma_m,\sigma_{m+1}}.
     \end{align*}
     Requiring this to be strictly less than $1$ makes $G$ a contraction.

\end{proof}

\begin{remark}
\label{remark:seq}
     $|\sigma|$ is always upper bounded by $j-i+1$, and $d_{i,j}$ has cardinality $2^{j-i-1}$ for $j>i$ ($|d_{i,i}|=1$), since it follows the sequence $a_{j+1} = \sum_{k=0}^j a_k$ with $a_0=1$, (A166444,~\cite{oeis}). This means the computation of the constant requires a number of operations that is growing exponentially with $p$.
\end{remark}

    The constants $K_{i,j}$ can be seen as defining a dependence structure matrix $(\mathcal{K})_{ij}=K_{i,j}$ between the equations; a nonzero value indicates a coupling, and a zero value indicates no direct coupling. If $\mathcal{K}$ has strictly positive elements on its subdiagonal and zeros elsewhere ($K_{i,j} > 0 \iff j=i-1$), then the problem is said to be \textit{linearly structured}; any equation depends only on the solution of the previous equation ($\mathbf{u}_i=\mathbf{u}_i(\mathbf{x},\mathbf{u}_{i-1}(\mathbf{x}))$). More generally, the matrix $\mathcal{K}$ is the adjacency matrix of a weighted directed acyclic graph (DAG). This reveals that the order in which the equations are solved is a topological sorting of the dependence graph, which exists if and only if the dependence can be modelled by a DAG. For a linearly structured problem, the graph is a simple path.

    In the case where the problem is the Picard method to solve a system of multiple coupled equations, the function $\varphi$ is
\begin{align*}
    \varphi(\mathbf{x},\mathbf{u}_1,\dots,\mathbf{u}_p) = (\mathbf{u}_1,\dots,\mathbf{u}_p),
\end{align*}
meaning $L_0=0$, $L_i=1$ ($i> 0$). $\varphi$ will be called a \textit{Picard solver}. The case where the first equation depends only on the last equation
\begin{align*}
    \varphi(\mathbf{x},\mathbf{u}_1,\dots,\mathbf{u}_p) = \mathbf{u}_p
\end{align*}
then $L_p = 1$, $L_i=0$ ($0\leq i<p)$. $\varphi$ will be called a \textit{weak Picard solver}.

An upper bound on the Lipschitz constant sufficient for the operator $G$ to be a contraction can be obtained.

\begin{corollary}
    Suppose that the Lipschitz constants satisfy $K_{i,j}\leq\kappa$ and $L_i\leq\lambda$. If any of the following conditions holds:
    \begin{enumerate}
        \item $\lambda(\kappa+1)^p < 1$,
        \item $\varphi$ is a Picard solver and $\kappa < 2^{\frac{1}{p}}-1$,
        \item $\mathcal{K}$ is linearly structured and
        $\frac{\kappa^{p+1}-\kappa}{\kappa-1}<\frac{1-\lambda}{\lambda}$,
        \item $\varphi$ is a Picard solver, $\mathcal{K}$ is linearly structured, $\kappa < 1$ and $\kappa^{p+1}-2\kappa + 1 > 0$ (in particular, $\kappa<\frac{1}{2}$ satisfies the constraints),
        \item $\varphi$ is a weak Picard solver, $\mathcal{K}$ is linearly structured and
        \begin{align*}
            \prod_{i=1}^{p} K_{i,i-1} < 1,
        \end{align*}
    \end{enumerate}
    then $G$ is a contraction.
\end{corollary}
\begin{proof}
For the first identity, the coefficient from~\Cref{thm:convergence} can be upper bounded by
\begin{align}
    \label{eq:eqbinom}
    L < \lambda\paren{1+\sum_{j=1}^p\sum_{\ell=0}^{j-1}{\binom{j-1}{\ell}}\kappa^{\ell+1}}
\end{align}
because the number of paths of length $\ell + 1$ between $j$ and $0$ corresponds exactly to $\binom{j-1}{\ell}$. By the binomial theorem and the geometric sum identity,
$$\lambda\qty(1+\sum_{j=1}^p \kappa (\kappa+1)^{j-1}) = \lambda\qty(1+ \kappa \frac{(\kappa+1)^p-1}{\kappa}) = \lambda (\kappa+1)^p $$
the formula is obtained.

By assuming the hypotheses of statement 2 on $L_i$, the expression of~\Cref{thm:convergence} becomes
\begin{align*}
    \sum_{j=1}^p\sum_{\sigma\in d_{0,j}}\prod_{m=1}^{\abs{\sigma}-1} K_{\sigma_m,\sigma_{m-1}} \leq \sum_{j=1}^p\sum_{\ell=0}^{j-1}{\binom{j-1}{\ell}}\kappa^{\ell+1}=(\kappa+1)^p - 1 < 1
\end{align*}
by using the same strategy as for the first statement. Isolating $\kappa$ yields the second expression.

For the third expression, this means that the only term left in~\eqref{eq:eqbinom} is
\begin{align}
    \label{eq:eqfraction}
    \lambda\qty(1+\sum_{j=1}^p\kappa^j) = \lambda\qty(1+\kappa\frac{\kappa^p - 1}{\kappa -1}) < 1.
\end{align}
The fourth expression is obtained by setting $\lambda =1$ in~\eqref{eq:eqfraction} and removing the first unit term (because $\varphi$ is a Picard solver), then writing it in polynomial form. $\kappa<\frac{1}{2}$ is sufficient because it ensures that for any $p$, the polynomial is positive. The last statement is obtained by considering only $j=p$ in the summation then simplifying.
\end{proof}
The fifth statement is well-known, for instance from~\cite{Traub_1964} (\S~2, Thm.~4), but can still be derived from the general bound of~\Cref{thm:convergence}, which can be applied in the case where the dependence structure is more complex.

This theorem is still valid in the infinite-dimensional case. As an example, consider the system of equations
\begin{align*}
    -\nabla\cdot( D_1 \nabla \mathbf{u}_1) = f_1(\mathbf{u}_1,\mathbf{u}_2)\\
    -\nabla\cdot( D_2 \nabla \mathbf{u}_2) = f_2(\mathbf{u}_1,\mathbf{u}_2)
\end{align*}
on a smooth bounded domain with homogeneous Dirichlet boundary conditions. It can be solved in an iterative fashion using the Picard solver $\varphi(\mathbf{x},\mathbf{u}_1,\mathbf{u}_2) = (\mathbf{u}_1,\mathbf{u}_2)$ with the scheme
\begin{align*}
    -\nabla\cdot( D_1 \nabla \mathbf{u}_1^{k+1}) &= f_1(\mathbf{u}_1^k,\mathbf{u}_2^k),\\
    -\nabla\cdot( D_2 \nabla \mathbf{u}_2^{k+1}) &= f_2(\mathbf{u}_1^{k+1},\mathbf{u}_2^k),
\end{align*}
where $\mathbf{x}^k = (\mathbf{u}_1^k,\mathbf{u}_2^k)$. Since $p=2$, the problem is necessarily linearly structured. We can show using the Poincaré and Cauchy-Schwarz inequalities and the mean value theorem that
\begin{align*}
    \|\mathbf{u}_1(\mathbf{x}) - \mathbf{u}_1(\mathbf{z})\|_{H^1_0} \leq \underbrace{ C_P^2\frac{\|\nabla_\mathbf{u} f_1(\xi(\mathbf{x},\mathbf{z}))\|_\infty}{\min_\Omega D_1}}_{K_{1,0}} \| \mathbf{x} - \mathbf{z}\|_{(H^1_0)^2},\\
    \|\mathbf{u}_2(\mathbf{x},\mathbf{u}_1) - \mathbf{u}_2(\mathbf{z},\mathbf{u}_1)\|_{H^1_0} \leq \underbrace{ C_P^2\frac{\|\partial_{\mathbf{u}_2} f_2(\mathbf{u}_1,\xi(\mathbf{x}_2,\mathbf{z}_2))\|_\infty}{\min_\Omega D_2}}_{K_{2,0}} \| \mathbf{x} - \mathbf{z}\|_{(H^1_0)^2},\\
    \|\mathbf{u}_2(\mathbf{x},\mathbf{u}_1) - \mathbf{u}_2(\mathbf{x},\mathbf{z}_1)\|_{H^1_0} \leq \underbrace{ C_P^2\frac{\|\partial_{\mathbf{u}_1} f_2(\xi(\mathbf{u}_1,\mathbf{z}_1),\mathbf{x}_2)\|_\infty}{\min_\Omega D_2}}_{K_{2,1}} \| \mathbf{u}_1 - \mathbf{z}_1\|_{H^1_0}.
\end{align*}
where $C_P$ is the Poincaré constant for $\Omega$, and $\xi(a,b)=\lambda a+(1-\lambda)b$ for some $\lambda\in(0,1)$.
Thus,
\begin{align*}
    K_{i,j} \leq \kappa := \frac{C_P^2\sum_{i,j=1}^2  \|\partial_{\mathbf{u}_i} f_j\|_\infty}{\min(\min_\Omega D_1, \min_\Omega D_2)} ,
\end{align*}
assuming $f_i$ are differentiable and Lipschitz continuous in each of their arguments. Since this is a linearly structured Picard solver, it corresponds to the fourth statement. If this constant is smaller than $\frac{\sqrt{5}-1}{2}=\phi^{-1} \approx 0.619$ ($\phi$ is the golden ratio), then the fixed point problem has a unique solution.

\subsubsection{ROM Approximation}

Assume the $i$-th equation is solved using a PROM, meaning $\mathbf{u}_i \approx \widehat{\mathbf{u}}_i$, where $\widehat{\mathbf{u}}_i$ is obtained by solving~\eqref{eq:lowdimsys} with matrices $A_i$ and right-hand sides $F_i$. Notice that the parameter used for the dimensionality reduction is $\mu_i^{k+1}$, and the snapshots are obtained from data at previous iterations. This will define an inexact fixed point operator
\begin{align*}
    G_k(\mathbf{x}) = \varphi(\mathbf{x},\mathbf{u}_1,\dots,\widehat{\mathbf{u}}_i,\dots,\widehat{\mathbf{u}}_p).
\end{align*}
Notice that because of the dependence on previous solutions, all solutions $\mathbf{u}_j$ with $j>i$ are actually perturbed as well. The perturbed solutions are defined recursively in the following way:
\begin{align*}
    \widehat{\mathbf{u}}_{i+1} &= \mathbf{u}_{i+1}(\mathbf{x},\dots,\widehat{\mathbf{u}}_i),\\
    \widehat{\mathbf{u}}_{i+2} &= \mathbf{u}_{i+2}(\mathbf{x},\dots,\widehat{\mathbf{u}}_i,\widehat{\mathbf{u}}_{i+1}),\\
    &\vdots
\end{align*}
To measure the total error committed by using $G_k$ instead of $G$, the error must be propagated through the successive perturbed solutions.

\begin{theorem}
\label{eq:errROMCoupled}
Under the same hypotheses on $\varphi$ and $\mathbf{u}_j$ ($j=1,\dots,p$) as in~\Cref{thm:convergence}, if the approximation $\mathbf{u}_i \approx \widehat{\mathbf{u}}_i$ is made to obtain $G_k$, the sequence
    \begin{align*}
        \delta^{k+1}_{i,p} = \left(L_i+\sum_{j={i+1}}^p L_j\sum_{\sigma\in d_{i,j}}\prod_{m=1}^{\abs{\sigma}-1}K_{\sigma_m,\sigma_{m+1}}\right) \|A_i^{-1}(\mu_i^{k+1})\| \|r_i(\mu_i^{k+1},\widehat{\mathbf{u}}^{k+1}_i)\|
    \end{align*}
    is an upper bound on the approximation error in the sense of~\Cref{hyp:controled}.
\end{theorem}
\begin{proof}
    The proof is identical to that of the previous theorem, mutatis mutandis, starting from the inequalities
    \begin{align*}
    \abs{G(\mathbf{x})-G_k(\mathbf{x})} = \abs{\varphi(\mathbf{x},\mathbf{u}_1,\dots,\mathbf{u}_p) - \varphi(\mathbf{x},\mathbf{u}_1,\dots,\widehat{\mathbf{u}}_i,\dots,\widehat{\mathbf{u}}_p)}\leq \sum_{j=i}^p L_j \| \mathbf{u}_j - \widehat{\mathbf{u}}_j\|
    \end{align*}
    and, for $j>i$,
    \begin{align*}
        \norm{\mathbf{u}_j-\widehat{\mathbf{u}}_j} \leq \sum_{\ell = i}^{j-1} K_{j,\ell} \norm{\mathbf{u}_\ell-\widehat{\mathbf{u}}_\ell},
    \end{align*}
    and then by expressing every error in terms of $\norm{\mathbf{u}_i-\widehat{\mathbf{u}}_i}$. Since this last approximation is made with a ROM, the result follows from~\eqref{eq:errgeneralROM}.
\end{proof}

It is interesting to point out a few very common cases:
\begin{enumerate}
    \item In the case $p=1$, $i=1$ (necessarily) and
\begin{align*}
    \delta^k_{1,1} = L_1 \|A_1^{-1}(\mu_1^k)\| \|r_1(\mu_1^k,\widehat{\mathbf{u}}^k_1)\|.
\end{align*}
    \item In the case $p=2$ and $i=1$,
    \begin{align*}
        \delta^k_{1,2} = (L_1 + L_2K_{2,1})\|A_1^{-1}(\mu_1^k)\| \|r_1(\mu_1^k,\widehat{\mathbf{u}}^k_1)\|.
    \end{align*}
    \item In the case $p=2$ and $i=2$,
    \begin{align*}
        \delta^k_{2,2} = L_2\|A_2^{-1}(\mu_2^k)\| \|r_2(\mu_2^k,\widehat{\mathbf{u}}_2^k)\|.
    \end{align*}
\end{enumerate}

The case where multiple systems are approximated by ROMs can be simply derived.

\begin{corollary}

    Assume the same regularity conditions as in the previous theorems. In the case where multiple systems are approximated with a ROM, say $\mathbf{u}_{i}\approx\widehat{\mathbf{u}}_{i}$ for $i\in\mathcal{I}\subseteq\{1,\dots,p\}$, then the sequence
    \begin{align*}
        \delta^k_{\mathcal{I},p} = \sum_{i\in\mathcal{I}} \delta^k_{i,p}
    \end{align*}
    is an upper bound on the approximation error in the sense of~\Cref{hyp:controled}.
\end{corollary}
\begin{proof}
    The proof is made by induction on $\abs{\mathcal{I}}$. Assume that, for some $m$ such that $1\leq m<p$,
    \begin{align*}
        \delta^k_{\{i_1,\dots,i_m\},p} = \sum_{j=1}^{m}\delta^k_{i_j,p}.
    \end{align*}
    In the case where the index set is enlarged to $\mathcal{I}=\{i_1,\dots,i_{m+1}\}$, the approximation error is
    \begin{align*}
        \delta^k_{\{i_1,\dots,i_{m+1}\},p} = \|\varphi(\mathbf{x},\mathbf{u}_1,\dots,\mathbf{u}_p) - \varphi(\mathbf{x},\mathbf{w}_1,\dots,\mathbf{w}_{p})\|,
    \end{align*}
    where $\mathbf{w}_i$ are defined recursively as
    \begin{align*}
    \mathbf{w}_{i} =
        \begin{cases}
             \widehat{\mathbf{u}}_i(\mathbf{x},\mathbf{w}_1,\dots,\mathbf{w}_{i-1}) & i \in \mathcal{I},\\
           \mathbf{u}_i(\mathbf{x},\mathbf{w}_1,\dots,\mathbf{w}_{i-1}) & \text{otherwise}.
        \end{cases}
    \end{align*}
    Consider the term
    \begin{align*}
        \varphi(\mathbf{x},\mathbf{z}_1,\dots,\mathbf{z}_{p})
    \end{align*}
    with the following recursive definition
    \begin{align*}
    \mathbf{z}_{i} =
        \begin{cases}
             \widehat{\mathbf{u}}_i(\mathbf{x},\mathbf{z}_1,\dots,\mathbf{z}_{i-1}) & i \in \mathcal{I}\setminus\{i_{m+1}\},\\
           \mathbf{u}_i(\mathbf{x},\mathbf{z}_1,\dots,\mathbf{z}_{i-1}) & \text{otherwise},
        \end{cases}
    \end{align*}
    i.e.\ the same as $\mathbf{w}_i$ but ignoring the index $i_{m+1}$.
    This is equivalent to a fixed point step where the $i_{m+1}$-th system is solved with the FOM instead of the ROM (every occurrence of $\widehat{\mathbf{u}}_{i_{m+1}}$ becomes $\mathbf{u}_{i_{m+1}}$). The triangle inequality yields the following:
    \begin{align*}
        \|\varphi(\mathbf{x},\mathbf{u}_1,\dots,\mathbf{u}_p) - \varphi(\mathbf{x},\mathbf{w}_1\dots,\mathbf{w}_p)\| &\leq
        \|\varphi(\mathbf{x},\mathbf{u}_1,\dots,\mathbf{u}_p) -  \varphi(\mathbf{x},\mathbf{z}_1\dots,\mathbf{z}_p)\| \\
        &+ \|\varphi(\mathbf{x},\mathbf{z}_1\dots,\mathbf{z}_p) - \varphi(\mathbf{x},\mathbf{w}_1\dots,\mathbf{w}_p)\|.
    \end{align*}
    The first term is simply $\delta^k_{\{i_1,\dots,i_m\},p}$. Since the only difference in the second error term is the approximation of a single linear system ($\mathbf{w}_j=\mathbf{z}_j$ for $j<i_{m+1}$),~\Cref{eq:errROMCoupled} can be applied. By applying the induction hypothesis, the previous inequality becomes:
    \begin{align*}
        \|\varphi(\mathbf{x},\mathbf{u}_1,\dots,\mathbf{u}_p) - \varphi(\mathbf{x},\mathbf{w}_1\dots,\mathbf{w}_p)\| &\leq \delta^k_{\{i_1,\dots,i_m\},p} + \delta_{i_{m+1},p} = \sum_{j=1}^{m+1} \delta_{i_{j},p}.
    \end{align*}
    The base case ($\mathcal{I}=\{i_1\}$) is trivially satisfied by~\Cref{eq:errROMCoupled}. By induction, the property will hold for any $\mathcal{I}=\{i_1,\dots,i_m\}\subseteq \{1,\dots,p\}$.

\end{proof}
\begin{remark}
    Because of the definition of the $\mathbf{w}_i$, the residuals $r_i(\mu_i^{k+1}, \widehat{\mathbf{u}}_i^{k+1})$ in $\delta_{i,p}^{k+1}$ have to be evaluated with information computed at the current iteration, meaning that $\mu_i^{k+1}$ may contain both FOM and ROM solutions, depending on which solver was used for each subsystem at the current iteration.
\end{remark}

Again, a common case is worth pointing out. In the case where $p=2$ and both equations are solved approximately, the upper bound becomes
\begin{align*}
     \delta^k_{\{1,2\},2} = (L_1 + L_2K_{2,1})\|A_1^{-1}(\mu_1^k)\| \|r_1(\mu_1^k,\widehat{\mathbf{u}}_1)\| + L_2\|A_2^{-1}(\widehat{\mu}_2^k)\| \|r_2(\widehat{\mu}_2^k,\widehat{\mathbf{u}}^k_2)\|,
\end{align*}
where $\widehat{\mu}_2^k = (\mathbf{x}^{k-1},\widehat{\mathbf{u}}^k_1)$.

\section{Numerical Illustration}
\label{sec:application}

The algorithm will be applied to the solution of coupled quasi-linear PDEs, for instance the Navier-Stokes problem coupled with an advection-diffusion problem. This specializes the theory of nested inexact fixed points from~\cite{Birken_2015} to the particular case of PROMs.

Earlier, it was mentioned that the ROM-based inexact fixed point algorithm does not guarantee an acceleration in itself, but only that the error is more accurately approximated by $\eta^k$ than by using the pure residual; this will be illustrated in the following section. However, it is important to mention in which context does an acceleration truly happens. The first important remark is that a high-dimensional finite element matrix has to be assembled to construct the ROM (the reduced matrix is obtained by $V^\top A_i(\mu_i^k)V$) and to obtain the error estimator (it appears in the definition of the residual). Even if the total number of FOM solves might be globally decreased, the number of matrix and residual assembly is increased. Consequently, this computation becomes the new dominant cost of the algorithm, and it must not exceed the cost saved by replacing the FOM with a ROM.

\begin{hypothesis}
    The assembly cost of the finite element matrix and residual is small or negligible when compared to the cost of solving the FOM.
    \label{Hypo:smallcost}
\end{hypothesis}

There are two clear regimes in which this happens, and the algorithm will yield a speedup:
\begin{enumerate}
    \item The problem has a small to moderate number of degrees of freedom (ensuring matrix/residual assembly is very fast), but there is a strong nonlinearity, requiring a lot of iterations for the nonlinear solver to converge (i.e., the ROM solves replace many FOM solves).
    \item The problem has a very large number of degrees of freedom, requiring the use of parallelism to reach a solution in a reasonable time, which also ensures that matrix/residual assembly is very fast because this computation scales much better than the solution of a linear system. Replacing a FOM evaluation with a ROM evaluation also has a more significant effect.
\end{enumerate}
To illustrate the impact of parallelism on the observed speedups, the algorithm will be tested on a 2D heated lid-driven cavity problem, and then on its 3D variant, solved with various numbers of CPUs. The two-dimensional problem will not yield good speedups, but will allow one to validate the high-accuracy of the method and study its behaviour under changing parameters. The three-dimensional problem, on the contrary, will reach much better speedups.

\subsection{Alternate Solution of Coupled Equations}
\label{section:alternatesolution}

In a general case, we could be interested in solving two quasi-linear partial differential equations
\begin{align*}
     \Lambda_1(\mathbf{u}_1,\mathbf{u}_2)\mathbf{u}_1 &= f_1(\mathbf{u}_1,\mathbf{u}_2),\\
        \Lambda_2(\mathbf{u}_1,\mathbf{u}_2) \mathbf{u}_2 &= f_2(\mathbf{u}_1,\mathbf{u}_2),
\end{align*}
where $\Lambda_1(\cdot,\cdot)$ and $\Lambda_2(\cdot,\cdot)$ are elliptic linear differential operators. Using a fixed point scheme, the solutions would be obtained in an alternating manner, given $\mathbf{u}_1^0$ and $\mathbf{u}_2^0$:
\begin{align*}
\Lambda_1(\mathbf{u}_1^{k+1},\mathbf{u}_2^{k})\mathbf{u}_1^{k+1} &= f_1(\mathbf{u}_1^{k+1},\mathbf{u}_2^{k}),\\
        \Lambda_2(\mathbf{u}_1^{k+1},\mathbf{u}_2^{k+1}) \mathbf{u}_2^{k+1} &= f_2(\mathbf{u}_1^{k+1},\mathbf{u}_2^{k+1}).
\end{align*}
A linearization step~\citep{Thomée_2006} can be applied under suitable regularity assumptions to reduce the complexity of the problem, by requiring that only linear systems be solved at each step:
\begin{align*}
\Lambda_1(\mathbf{u}_1^{k},\mathbf{u}_2^{k})\mathbf{u}_1^{k+1} &= f_1(\mathbf{u}_1^{k},\mathbf{u}_2^{k}),\\
        \Lambda_2(\mathbf{u}_1^{k+1},\mathbf{u}_2^{k}) \mathbf{u}_2^{k+1} &= f_2(\mathbf{u}_1^{k+1},\mathbf{u}_2^{k}).
\end{align*}
Let the discretized linear equations be
  \begin{align*}
A_1(\mathbf{u}_1^{k},\mathbf{u}_2^{k})\mathbf{u}_1^{k+1} &= F_1(\mathbf{u}_1^{k},\mathbf{u}_2^{k}),\\
        A_2(\mathbf{u}_1^{k+1},\mathbf{u}_2^{k}) \mathbf{u}_2^{k+1} &= F_2(\mathbf{u}_1^{k+1},\mathbf{u}_2^{k}).
\end{align*}
    obtained, for instance, by FEM (from now on, by abuse of notation, the solutions $\mathbf{u}_1^{k}$ will be considered as discretized).

       \begin{hypothesis} \label{Hypo:NormeInv}
   The matrices $A_1$ and $A_2$ are invertible and their inverses are uniformly bounded on a subset containing all of the fixed point iterations, i.e.,
   \begin{equation*}
       \| A_1(\cdot,\cdot)^{-1} \|, \| A_2(\cdot,\cdot)^{-1} \| < M,
   \end{equation*}
   for some real constant $M>0$.
\end{hypothesis}

\begin{hypothesis} \label{Hypo:LipS} The solution operator $(\mathbf{x},\mathbf{u}_1)\mapsto \mathbf{u}_2$ is $K_{2,1}$-Lipschitz continuous w.r.t.\ $\mathbf{u}_1$ uniformly over $\mathbf{x}$,
   \begin{equation*}
       \|{\mathbf{u}_2(\mathbf{x},\mathbf{w}) - \mathbf{u}_2(\mathbf{x},\mathbf{z})}\| \leq K_{2,1} \|\mathbf{w} - \mathbf{z} \|,\quad \forall \mathbf{x},\mathbf{w},\mathbf{z}.
   \end{equation*}
\end{hypothesis}

    Following the notation of the previous section, note that this is equivalent to $\mathbf{x}^k=(\mathbf{u}_1^{k},\mathbf{u}_2^{k})$, $\mathbf{u}_1^{k+1}=\mathbf{u}_1(\mathbf{x}^k)$ and $\mathbf{u}_2^{k+1} = \mathbf{u}_2(\mathbf{x}^k,\mathbf{u}_1(\mathbf{x}^k))$. The fixed point iteration to obtain the simultaneous solution to both equations is given by
    \begin{align*}
    \varphi(\mathbf{x},\mathbf{u}_1, \mathbf{u}_2) =  (\mathbf{u}_1, \mathbf{u}_2).
    \end{align*}
   \begin{remark}
    The function $\varphi$ depends on the choice of fixed point scheme chosen to solve the problem, and not on the problem itself.
\end{remark}
The function $\varphi$ is entrywise Lipschitz continuous with constants $L_0=0$ and $L_1=L_2=1$. As a consequence of~\Cref{eq:errROMCoupled}, if a ROM is used on the first equation (special case 2), the approximation error is upper bounded by
\begin{align*}
        \delta^k = (1+K_{2,1})M \|r_1(\mu_1^k,\widehat{\mathbf{u}}_1^k)\|.
\end{align*}
If a ROM is used on the second equation (special case 3), then
\begin{align*}
        \delta^k = M \|r_2(\mu_2^k,\widehat{\mathbf{u}}_2^k)\|.
\end{align*}
Finally, if the ROM is used on both equations,
\begin{align*}
        \delta^k = (1 + K_{2,1})M\|r_1(\mu_1^k,\widehat{\mathbf{u}}_1)\| + M \|r_2(\widehat{\mu}_2^k,\widehat{\mathbf{u}}_2^k)\|.
\end{align*}

\begin{remark}
    The norm used to compute any quantity in this section is the Euclidean norm on the finite element vector of degrees of freedom (DoFs). The distance mentioned in~\Cref{cor:main} will be the metric induced by this norm.
\end{remark}

\subsection{Constants Estimation}
The error estimation depends on multiple constants related to the regularity of the data. These constants can be hard to determine analytically for complex nonlinear problems. However, they can be estimated from iterations of the fixed point scheme.

The quantity $\|A_1^{-1}\|$ is required, which can be computed 
efficiently via power iterations when the matrix is already factorized to solve the FOM linear system. However, the upper bound $\|A_1^{-1}\|\|r_1\|$ is typically a severe overestimate of the true error 
$\|u_h - \widehat{u}\|$. When 
this bound is used in the algorithm, $\delta^k$ is so 
pessimistic that ROM steps are rejected at virtually every 
iteration, collapsing the algorithm to a pure FOM scheme and 
eliminating any computational benefit. Instead, the quantity $\|A_1^{-1}\|$ is replaced by the tighter 
estimate
\begin{align*}
  M \approx \frac{\|\mathbf{u}_1^{k+1}\|}{\|F_1(\mathbf{u}_1^k,
  \mathbf{u}_2^k)\|}.
\end{align*}
This estimate is not a 
provable upper bound on $\|A_1^{-1}\|$ and therefore places the 
algorithm in the practical regime of concluding remark of
\Cref{section:errorprop}. However, it reflects the 
actual magnitude of the error much more accurately, and the 
numerical experiments of \Cref{section:ROMQM} demonstrate that 
the resulting algorithm both achieves meaningful acceleration and 
produces a final solution whose error is within the user-defined 
tolerance and comparable to the reference high-fidelity solution.

$K_{2,1}$ can be estimated from successive iterations:
\begin{equation*}
     \|\mathbf{u}_2^{k+1}-\mathbf{u}_2^k\| \leq K_{2,1} \|\mathbf{u}_1^{k+1}-\mathbf{u}_1^k\| \implies K_{2,1} \approx
    \frac{\|\mathbf{u}_2^{k+1}-\mathbf{u}_2^k\|}{\|\mathbf{u}_1^{k+1}-\mathbf{u}_1^k\|}.
\end{equation*}
This is equivalent to computing an approximate derivative of the map $\mathbf{u}_1 \mapsto \mathbf{u}_2$ using a forward finite difference, with an error $O(\|\mathbf{u}_1^{k+1}-\mathbf{u}_1^k\|)$. Again, taking the maximum over all iterations where this value is computed gives an estimate for $K_{2,1}$.

\subsection{2D Heated Lid-Driven Cavity}

The implementation is done in Python, using the finite element software MEF++ and its Python interface mefpp4py, in combination with petsc4py~\citep{dalcinpazklercosimo2011} and slepc4py~\citep{Hernandez_2005}. The main interest in using these libraries is their parallel computing capabilities.
The SVD solver used in parallel is the thick-restart Lanczos algorithm from SLEPc. The linear systems are all solved using $LU$ decompositions (the factorization cannot be reused between iterations since both matrices have a dependence on the unknowns). Moreover, when a new snapshot replaces an old one in the snapshot matrix, the SVD is updated incrementally using the algorithm for low-rank updates from \cite{Brand_2006}, and adapted in parallel following the approach of \citep{Kuhl_Fischer_Hinze_Rung_2024}. The computation requires only a local, low-dimensional SVD, which is then broadcast to the other processes. Because of the propagation of truncation error in the algorithm, the full SVD is recomputed once every few incremental calls to restart the incremental SVD.

The methodology is tested on a multiphysics simulation. The incompressible, non-Newtonian, steady Navier-Stokes equations are coupled to heat diffusion and convection phenomena. The problem studied is the heated lid-driven cavity flow \cite{Dumon_Allery_Ammar_2013}, a variant of the well-studied lid-driven cavity flow (e.g., \citep{Sahin_Owens_2003}).  The domain $\Omega$ is a unit square with the left wall $\Gamma_\mathrm{left}$ heated at $\theta = 1$ and the right wall $\Gamma_\mathrm{right}$ heated at $\theta=0$, the bottom wall $\Gamma_\mathrm{bottom}$ and top lid $\Gamma_\mathrm{top}$ are both adiabatic (\Cref{fig:geometry}). A no-slip condition is enforced on the left, right and bottom walls, and a tangent velocity $\mathbf{u}(x,y) = (1,0)$ is imposed at the lid. The equations and boundary conditions are
\begin{align}
&\begin{cases}
     \mathbf{u}\cdot \nabla \mathbf{u} -\nabla \cdot (\nu(\mathbf{u},\theta)\nabla \mathbf{u}) + \nabla p = f(\theta) &\qq{in} \Omega\\
     \nabla\cdot \mathbf{u} = 0 &\qq{in} \Omega\\
     \mathbf{u} =0 &\qq{on}\Gamma_\mathrm{left}\cup \Gamma_\mathrm{right} \cup \Gamma_\mathrm{bottom}\\
     \mathbf{u}_x = 1 &\qq{on}\Gamma_\mathrm{top} \\
     \mathbf{u}_y = 0 &\qq{on}\Gamma_\mathrm{top}
\end{cases}
\label{eq:stokes}
\\
&\begin{cases}
    -D\Delta \theta + \mathbf{u}\cdot\nabla \theta = 0 &\qq{in} \Omega\\
    \theta = 1 &\qq{on}\Gamma_\mathrm{left}\\
    \theta = 0  &\qq{on}\Gamma_\mathrm{right}\\
    \nabla\theta \cdot n = 0  &\qq{on}\Gamma_\mathrm{top} \cup \Gamma_\mathrm{bottom}\\
\end{cases}
\label{eq:heat}
\end{align}

\begin{figure}[ht]
    \centering
\begin{tikzpicture}[scale=0.6]

% --- Main square ---
\def\L{6}  % side length

\draw[line width=1.2pt] (0,0) rectangle (\L,\L);

% % --- Dashed orange cross (center lines) ---
% \draw[dashed, orange!80!red, line width=1pt] (\L/2, 0) -- (\L/2, \L);
% \draw[dashed, orange!80!red, line width=1pt] (0, \L/2) -- (\L, \L/2);

% --- Top boundary: inlet flow arrows and label ---
% Arrows above the top wall
\foreach \x in {0.6, 1.5, 2.5, 3.6, 4.7, 5.4} {
    \draw[->, line width=0.8pt] (\x, \L+0.3) -- (\x+0.55, \L+0.3);
}
\node[above] at (\L/2, \L+0.55) {$\mathbf{u}_x = 1$ \quad and \quad $\mathbf{u}_y = 0$};

% --- Bottom boundary: no-slip label ---
\node[below=0.2cm] at (\L/2, 0) {no slip};

% % Bottom dimension arrow
% \draw[<->, line width=0.8pt] (0, -1.1) -- (\L, -1.1);
% \node[below] at (\L/2, -1.1) {$d$};

% --- Left boundary: no-slip label ---
\node[left=0.2cm, align=center] at (0, \L/2) {no slip};

% --- Right boundary: no-slip label ---
\node[right=0.2cm, align=center] at (\L, \L/2) {no slip};

% % Right dimension arrow
% \draw[<->, line width=0.8pt] (\L+1.5, 0) -- (\L+1.5, \L);
% \node[right] at (\L+1.5, \L/2) {$d$};

% --- Interior labels ---
% Top-right quadrant: adiabatic
\node at ({\L/2}, {7*\L/8}) {adiabatic};

% Bottom-right quadrant: adiabatic
\node at ({\L/2}, {\L/8}) {adiabatic};

% % Top-left quadrant: hot
% \node[align=left] at ({\L/4}, {3*\L/4}) {hot \\ $T = T_h$};

% Bottom-left quadrant (mirrored): also hot region label sits in upper-left
% Actually hot label with T=Th is left half, cold is right half
% Re-reading image: left half is hot T=Th, right half is cold T=Tc
% hot label in left half, near horizontal centerline
\node[align=left] at ({\L/8}, {\L/2+0.35}) {hot};
\node[align=left] at ({\L/8}, {\L/2-0.35}) {$\theta = 1$};

% cold label in right half, near horizontal centerline
\node[align=right] at ({7*\L/8}, {\L/2+0.35}) {cold};
\node[align=right] at ({7*\L/8}, {\L/2-0.35}) {$\theta = 0$};

% --- Gravity arrow (top-left, outside) ---
\draw[->, line width=1pt] (-1.8, \L-0.5) -- (-1.8, \L-1.4);
\node[right] at (-1.8, \L-0.5) {$\mathbf{g}$};

\node[left] at (0,0) {$(0,0)$};

\node[right] at (\L,\L) {$(1,1)$};

% --- Coordinate axes (bottom-left corner, inside) ---
% \draw[->, line width=0.8pt] (0.1, 0.1) -- (0.9, 0.1);
% \node[below right] at (0.9, 0.1) {$x$};
% \draw[->, line width=0.8pt] (0.1, 0.1) -- (0.1, 0.9);
% \node[above left] at (0.1, 0.9) {$y$};

\end{tikzpicture}
\caption{The cavity's geometry and prescribed boundary conditions.}%
    \label{fig:geometry}
\end{figure}

The right-hand side of the momentum equation satisfies the Boussinesq approximation
 \begin{equation*}
     f(\theta) = -\beta\mathbf{g}\theta
 \end{equation*}
with $\mathbf{g}$ the gravitational acceleration vector and $\beta = 10 / |\mathbf{g}|$. The viscosity follows a Carreau law with an exponential dependance on temperature, i.e.,
\begin{equation*}
    \nu(\mathbf{u},\theta)= A\exp\bigg( \frac{B}{\theta+C}\bigg) \big( \nu_\infty + (\nu_0 - \nu_\infty)(1 + \lambda \| \nabla \mathbf{u}\|^{2})^{\frac{n-1}{2}} \big) \label{eq:visco}
\end{equation*}
The fluid is assumed to be shear-thinning, i.e., $\nu_\infty < \nu_0$. The values of these parameters are
\begin{equation*}
    A = 0.05,\quad
    B = 4,\quad
    C = 1,
\end{equation*}
\begin{equation*}
    \nu_\infty = 0.05,\quad
    \nu_0 = 0.5,\quad
    \lambda = 500,\quad
    n=0.3.
\end{equation*}
The heat diffusion coefficient in \eqref{eq:heat}, assumed to be constant, is $D = 0.01$.
According to the ranges of temperature and velocity, and to the parameters and geometry defined above, the dimensionless numbers characterizing the flow are 
 \begin{align*}
     \mathrm{Re}\sim 10^1 - 10^2 ,\quad \mathrm{Ri} =10, 
     \quad\mathrm{Pr}\sim 10^0-10^1,\quad \mathrm{Gr}= 10^3 - 10^5,
 \end{align*}
 respectively the Reynolds, Richardson, Prandtl, and Grashof numbers. These parameters (especially $\mathrm{Pr>1}$ and shear-thinning) can emulate the flow of non-Newtonian polymer solutions, which can be used to model some plastic recycling processes \citep{Savvas_Markatos_Papaspyrides_1994}.

The equations are non-linear, so an Oseen linearization \cite{John_2016} is applied to the convective term, and a semi-linearization is applied to the viscosity to make the system linear at each iteration, resulting in the following fixed point scheme:
\begin{align*}
    A_1\mathbf{u}_1^{k+1} = F_1 \leadsto &\begin{cases}
        \mathbf{u}^k\cdot \nabla \mathbf{u}^{k+1} -\nabla \cdot (\nu(\mathbf{u}^k,\theta^k)\nabla \mathbf{u}^{k+1}) + \nabla p^{k+1} = f(\theta^k),\\
        \nabla \cdot \mathbf{u}^{k+1} = 0,
    \end{cases}
    \\
    A_2 \mathbf{u}_2^{k+1} = F_2\leadsto &-D\Delta \theta^{k+1} + \mathbf{u}^{k+1}\cdot\nabla \theta^{k+1} = 0.
\end{align*}

% and the viscosity follows a three-parameter exponential law
%  \begin{align*}
%      \nu(\theta) = a \exp\qty({\frac{b}{\theta - c}}),
%  \end{align*}
% with $a = 0.005$, $b = 20$ and $c = -9$ (values within 0.02 and 0.04). The inlet velocity is set to a parabolic flow $\mathbf{u}_\mathrm{in}(x,y)=(0,1.8x(2-x))^\top$, and the inlet temperature and wall heat flux are set respectively to $\theta_\mathrm{in}=0$ and $\theta_\mathrm{wall}=0.12$. The thermal diffusion coefficient is set to $k_T = 0.04$. Those values generate a flow characterized by the following adimensional quantities:

% respectively the Reynolds, Richardson and Prandtl numbers. Numerically, it is validated that the fixed point operator $G$ arising from this setting is indeed a contraction, with Lipschitz constant $L\approx 0.9$.

The finite element method is used to discretize this system of PDEs. The momentum and incompressibility equations are coupled using the usual saddle-point formulation~\citep{Deteix_Diop_Fortin_2022}, with Taylor-Hood elements. Note that the variable $\mathbf{u}_1$ from~\Cref{section:alternatesolution} will represent the unknown of the monolithic mixed system, which corresponds to $\mathbf{u}_1=(\mathbf{u}_h,p_h)$. $P_2$ elements are used to discretize the temperature $\mathbf{u}_2=\theta_h$. Velocity and pressure combined have a total of 159\,560 DoFs and the temperature has 70\,849 DoFs.

First, a reference solution is obtained via standard Picard iterations with the high-fidelity solver and an absolute tolerance of $\varepsilon = 10^{-4}$ (which corresponds roughly in this case to a relative tolerance of $2.2\times 10^{-7}$). When the ROM is only used on~\eqref{eq:stokes}, it will be labelled U-ROM (resp.\ T-ROM for~\eqref{eq:heat}). Because of the poor performance of T-ROM in terms of speedup (\Cref{Hypo:smallcost} is not respected), the mixed velocity-temperature ROM will not be studied.

An approximate converged solution with its pointwise error map (with respect to the FOM solution) can be seen in~\Cref{fig:solutionU,fig:solutionT}. The magnitude of the error on the pressure field was of the same order. The absolute errors in the vector 2-norm are
\begin{align*}
    \| \mathbf{u}_\mathrm{ROM} - \mathbf{u}_\mathrm{ref} \| = 4.445 \times 10^{-6},\\
    \| \theta_\mathrm{ROM} - \theta_\mathrm{ref} \| = 1.996 \times 10^{-6},
\end{align*}
which are both two orders of magnitude below the set tolerance $\varepsilon = 10^{-4}$.

 \begin{figure}[htbp]
     \centering
     \subfloat[Velocity field and streamlines.]{
    \includegraphics[trim={1in 1.2in 1.2in 1.2in},clip,width=0.45\linewidth]{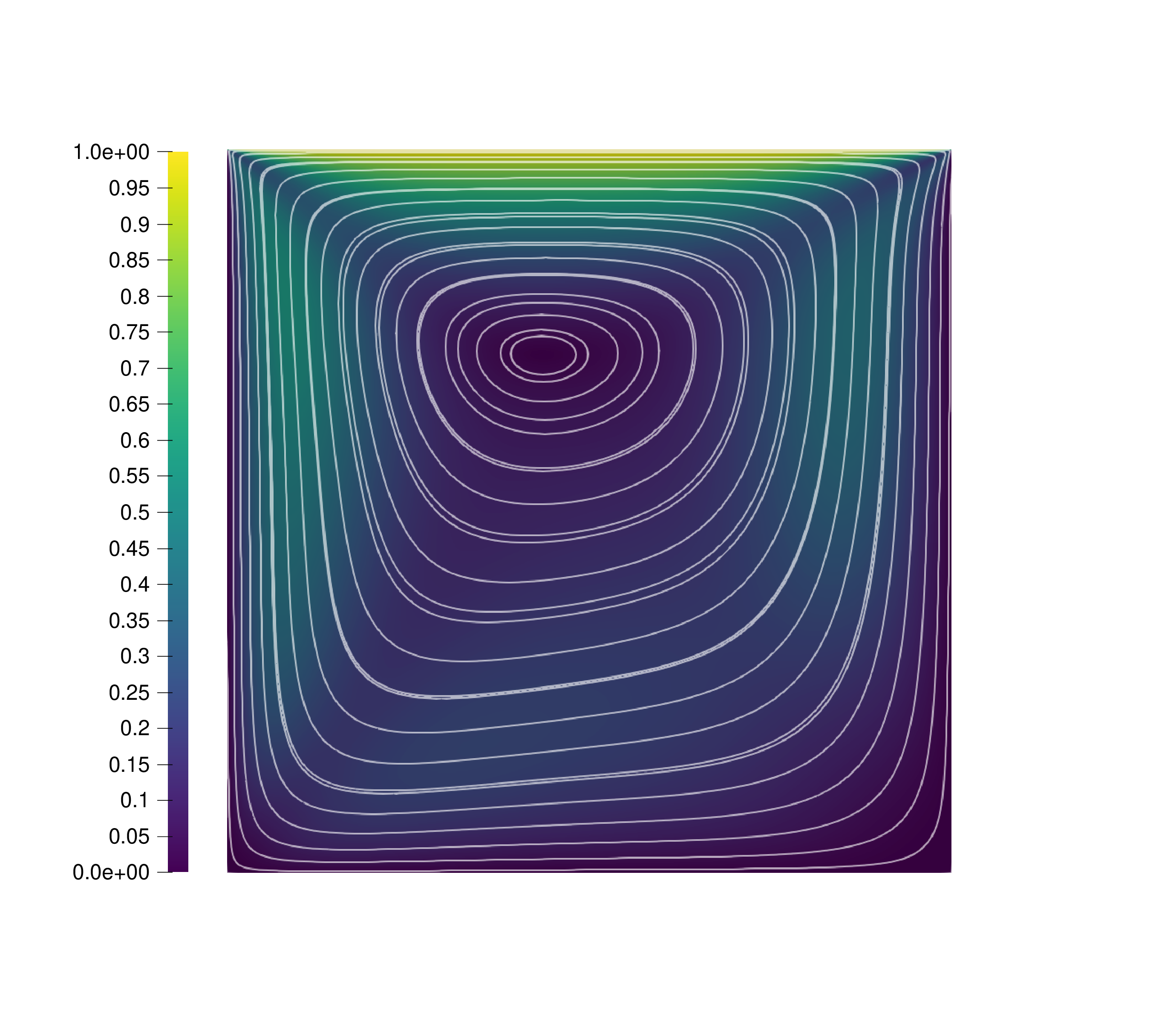}
    }%
     \subfloat[Pointwise error magnitude between the FOM and ROM final velocity field.]{
    \includegraphics[trim={1.2in 1.2in 1.2in 1in},clip,width=0.45\linewidth]{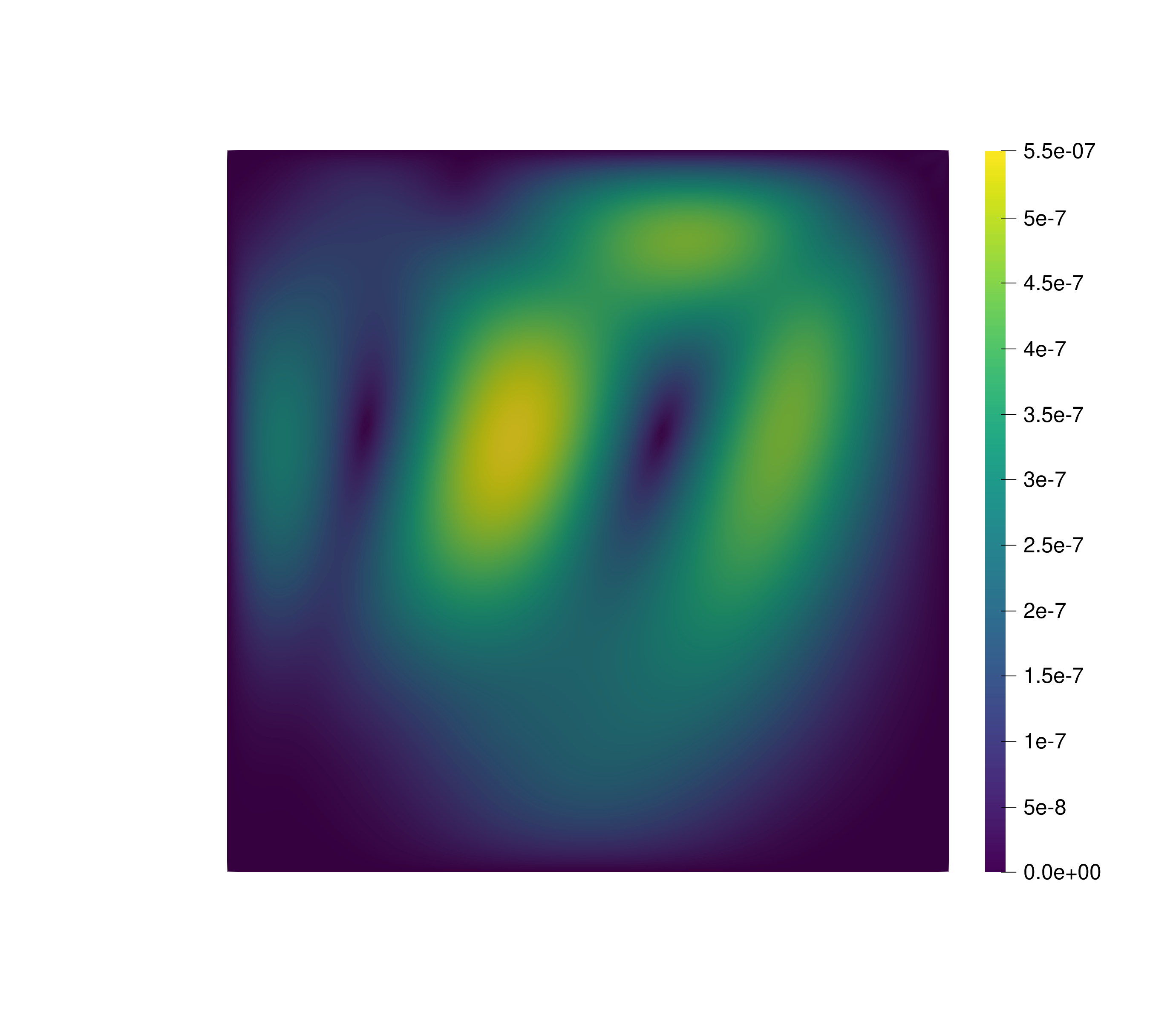}
     }
     \caption{Velocity field and errors w.r.t.\ the reference solution for the solution obtained with the U-ROM-accelerated fixed point scheme with $N_b=20$, $\varepsilon_\mathrm{rb}=10^{-7}$.}%
     \label{fig:solutionU}
 \end{figure}

  \begin{figure}[htbp]
     \centering
     \subfloat[Temperature field and isotherms.]{
    \includegraphics[trim={1in 1.2in 1.2in 1.2in},clip,width=0.45\linewidth]{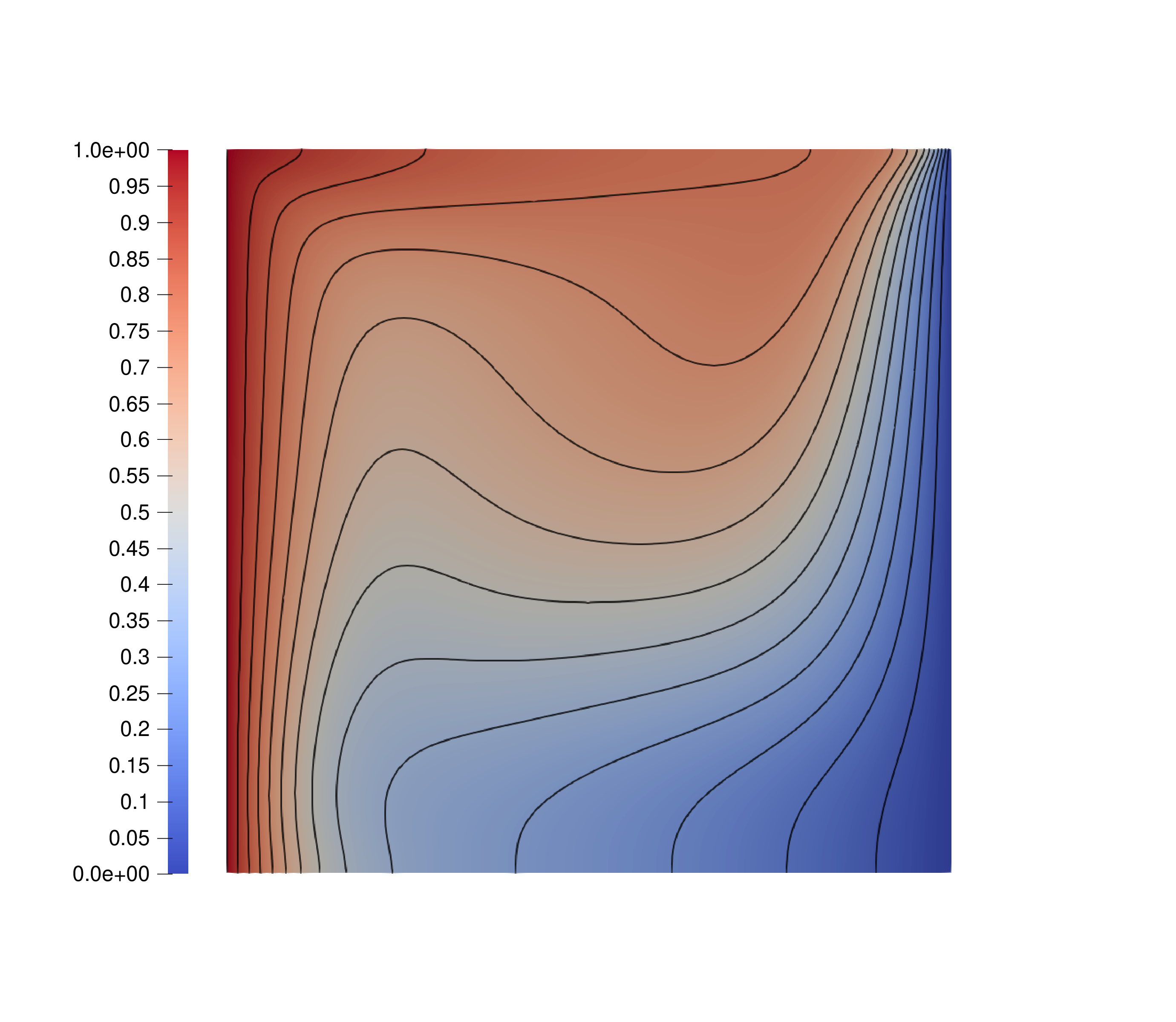}
    }%
     \subfloat[Pointwise error magnitude between the FOM and ROM final temperature field.]{
    \includegraphics[trim={1.2in 1.2in 1.2in 1in},clip,width=0.45\linewidth]{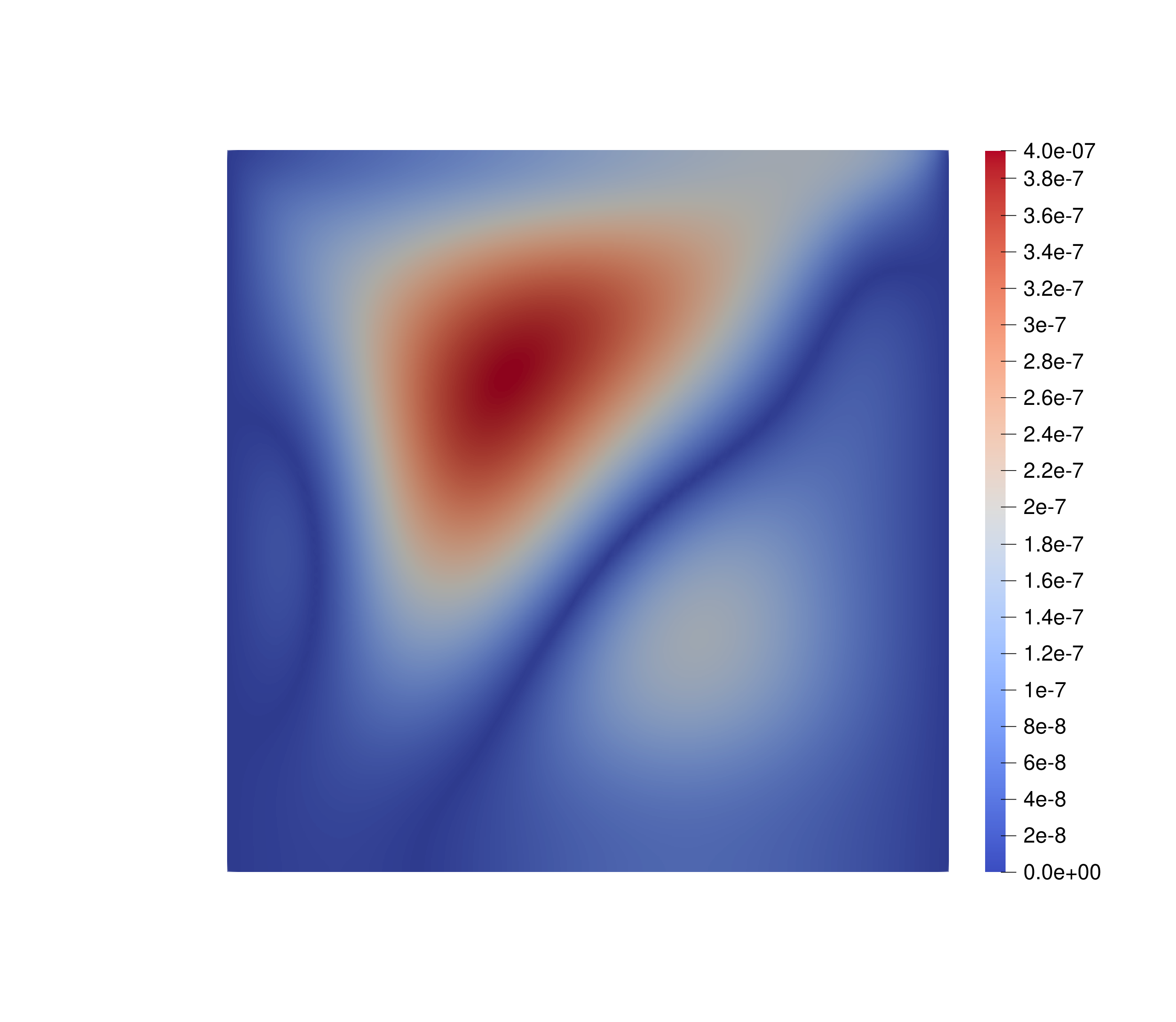}
     }
     \caption{Temperature field and error w.r.t.\ the reference solution for the solution obtained with the U-ROM-accelerated fixed point scheme with $N_b=20$, $\varepsilon_\mathrm{rb}=10^{-7}$.}%
     \label{fig:solutionT}
 \end{figure}

The behaviour of the error estimator along the iterations can be seen in~\Cref{fig:properror1,fig:properror2}. This allows one to visualize when the algorithm uses the FOM (when it is above the $\varepsilon$ line) and the ROM (when it is below). We see that for the U-ROM, the algorithm is very stable to changes in parameters. It always reaches a stable convergence regime where only the ROM is used. For the T-ROM, when decreasing $N_b$ or increasing $\varepsilon_\mathrm{rb}$ too much, convergence is harder to reach and FOM iterations/restarts by the outer validation loop are needed late in the process. \Cref{fig:convrate} shows that for the U-ROM, the convergence rate of the fixed point is maintained ($L\approx 0.8825$). For the T-ROM, we notice that the convergence rate is similar on average but there are more instabilities; a stable convergence regime is not reached, or is reached later in the process. As mentioned earlier, reducing $N_b$ seems to impact a lot the quality of the T-ROM, which can reflect on the stability of the convergence.

The behaviour of the algorithm differs between the two physical systems, but only when the hyperparameters are pushed to extremal values. Throughout its normal operating range the T-ROM converges to the same solution as the high-fidelity scheme, with a final error below the prescribed tolerance: the safeguard of \cref{algo:acceleratedscheme} correctly reverts to the FOM whenever the ROM is inadequate, so accuracy is never compromised. The difference is one of convergence rate. When $N_b$ is reduced or $\varepsilon_\mathrm{rb}$ increased too aggressively, the T-ROM does not settle into a stable regime in which the ROM is used predominantly.
We were not able to determine the mechanism behind this sensitivity, and a full diagnosis lies outside the scope of the present work, whose aim is to introduce and validate the propagation-based estimator rather than to characterize every regime in which acceleration may fail to materialize. We emphasize that this degradation is confined to extremal hyperparameter settings and is not observed for the U-ROM under any configuration tested. Several factors could conceivably contribute to the observed behavior (e.g., the strongly advective character of the temperature equation, known to be challenging for linear ROMs \citep{Geelen_Wright_Willcox_2023}) but we have not confirmed any of these as the cause and draw no conclusion. For these reasons, and because the temperature system is moreover less than half the size of the velocity system (so that limited time gains would be expected in any case), the T-ROM is excluded from the experiments that follow.

% This difference in the behaviour of the algorithm between the two physical systems is not yet fully understood; further investigation would be necessary. However, it might be related to the fact that the second equation has a strong advective component, which is known to perform more poorly with a linear ROM \citep{Geelen_Wright_Willcox_2023}. It might also be that the estimates of the constants are not good enough in the 2-norm for the second equation, and a more appropriate choice of the norm would allow for the estimator to capture the error more precisely. In further experiments, the T-ROM will be left out because of these unexplained instabilities. Moreover, the system is less than half the size of the velocity system, so time gains might not be expected.

\begin{figure}[ht]
    \centering
    \includegraphics[width=0.45\linewidth]{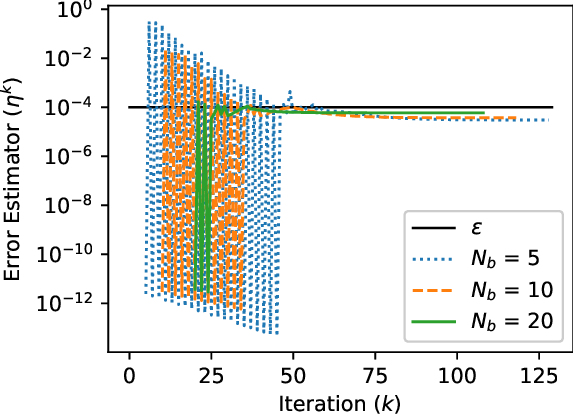}
    \qquad
    \includegraphics[width=0.45\linewidth]{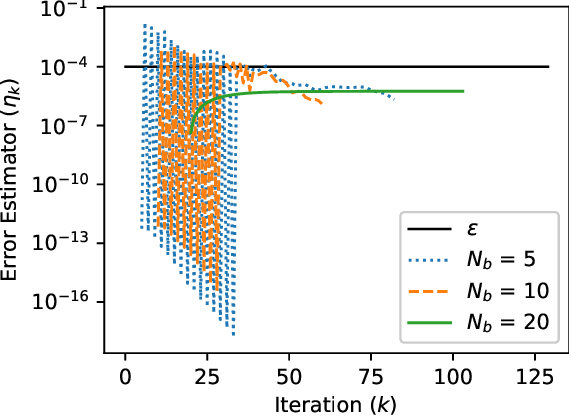}
    \caption{Propagated error $\eta^k$ along iterations for U-ROM (left) and T-ROM (right), for different maximal basis sizes $N_b$. $\varepsilon_\mathrm{rb}=10^{-7}$ is fixed.}%
    \label{fig:properror1}
\end{figure}
\begin{figure}[ht]
    \centering
    \includegraphics[width=0.45\linewidth]{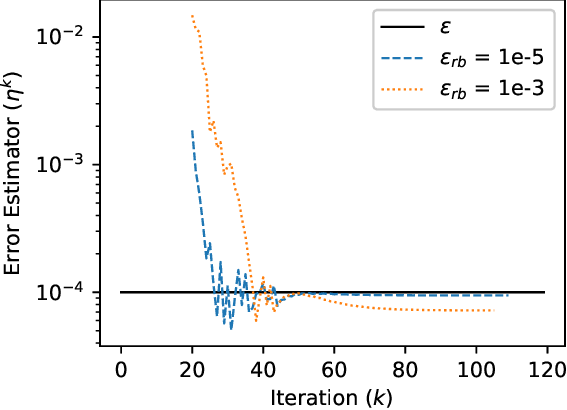}
    \qquad
    \includegraphics[width=0.45\linewidth]{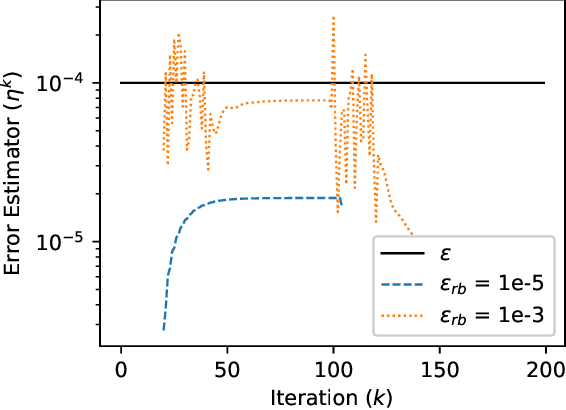}
    \caption{Propagated error $\eta^k$ along iterations for U-ROM (left) and T-ROM (right), for different reduced basis tolerances $\varepsilon_\mathrm{rb}$, with $N_b=20$.}%
    \label{fig:properror2}
\end{figure}

\begin{figure}[ht]
    \centering
    \includegraphics[width=0.45\linewidth]{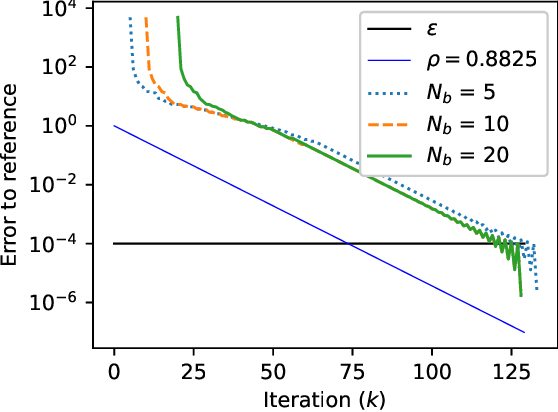}
    \qquad
    \includegraphics[width=0.45\linewidth]{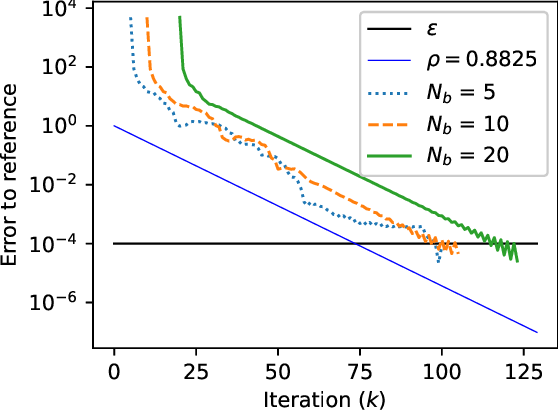}
    \caption{Error with respect to FOM solution along iterations for U-ROM (left) and T-ROM (right), for different reduced basis tolerances $\varepsilon_\mathrm{rb}$. The observed convergence rate for the FOM solution was $L=0.8825$.}%
    \label{fig:convrate}
\end{figure}

 All computation times are obtained by running the different algorithms on the computer cluster Fir (Digital Research Alliance of Canada), on a dedicated node equipped with 192 AMD EPYC 9655 (Zen 5) @ 2.7 GHz processors and 750GB of RAM. The runtimes have been averaged over 30 runs, and bootstrap confidence intervals have been computed. This allows to average out any influence of the cluster's hardware (e.g., disk I/O, network communication). A summary of the global trends is presented here. The main drivers of the computational cost are the FOM solution and the FE matrix assembly; in the case of the U-ROM, the FOM solver is dominant because of the size of the matrix. Performance metrics are shown in~\Cref{fig:speedupsbarplot}. First, it can be noticed that the ROM with $N_b=5$ does not accelerate the solution significantly, or even slows it down, until reaching 16 CPUs. This can be explained by the large number of switches between FOM and ROM iterations needed to capture correctly the trend of the model, which increase the total number of iterations required to reach convergence and the number of matrix assemblies. However, for $N_b=20$, we notice that for 2 or more CPUs, a significant speedup is observed. This means that, starting from 2 CPUs, the overhead cost of matrix assembly becomes less significant than the cost of the direct solver for this problem size. 
 
 \begin{figure}[htbp]
    \centering
    \includegraphics[width=0.45\linewidth]{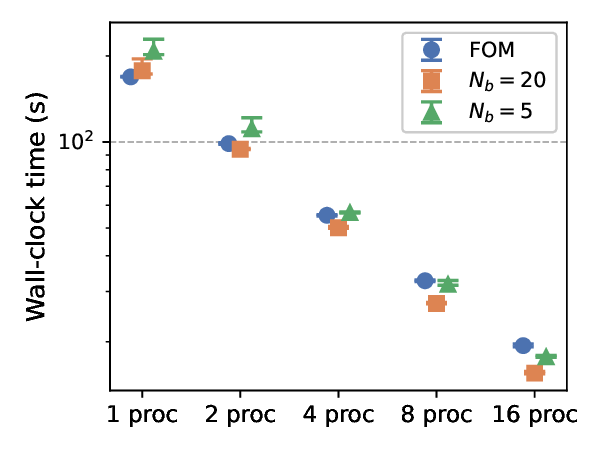}
    \qquad
    \includegraphics[width=0.45\linewidth]{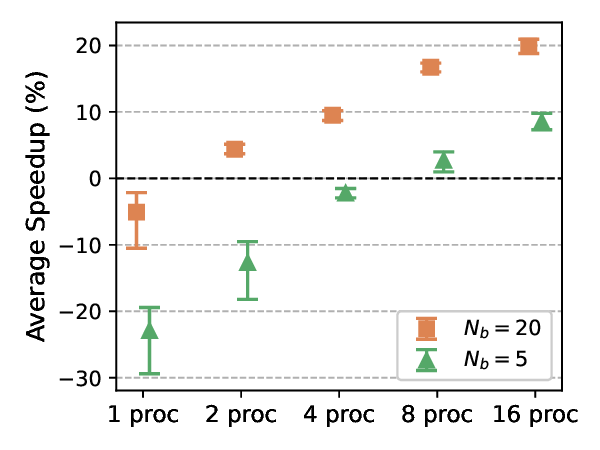}
    \caption{Wall-clock times (left, in log scale) and corresponding average speedups (right) for each variant of the algorithm, on multiple CPUs.}%
    \label{fig:speedupsbarplot}
\end{figure}

 Detailed execution times breakdown are shown in \Cref{tab:perf2cpu,tab:perf16cpu}. ``Proj.'' refers to the projection of the large matrix by the basis matrix $V$, and SVD accounts for incremental and full SVD calls. It can be observed that the incremental SVD approach is very efficient and scales well in parallel, such that the total SVD cost is negligible. However, since the problem is relatively small, the speedups are still limited; the three-dimensional case that will be presented in \Cref{sec:3DCavity} will show that the algorithm performs better when the problem is much more higher-dimensional.

\begin{table}[ht]
    \centering
    \caption{Performance metrics for the fixed point iterations using FOM and ROM for $N_b=5$ and $N_b=20$, with $\varepsilon_\mathrm{rb}=10^{-7}$, when using 2 CPUs.}
    {
    \begin{tabular}{@{}lccccccccc@{}}
        \toprule
         {Method} &  {Metric} &  {Iter.} &  {U/T-Assem.} &  {U/T-FOM} &  {ROM} &  {Proj.} &  {SVD} &  {$P$}\\
        \midrule
          {Ref.} & \makecell{ {\#} \\  {Time (s)}} &
          \makecell{104 \\ 104.3} &
          \makecell{104/104 \\ 35.2/7.83}  &
          \makecell{104/104 \\ 46.6/5.32} &
          --  &
         -- &
         -- &
         --\\
         \midrule
          {$N_b=20$} & \makecell{ {\#} \\  {Time (s)}} &\makecell{109 \\ \textbf{98.7}} &
         \makecell{ 280/105 \\ \textbf{61.4}/7.83 } &
         \makecell{ 27/105 \\ \textbf{12.7}/5.59 } &
         \makecell{ 84 \\ 0.033 }  &
         \makecell{ 84 \\ 2.42 } &
         \makecell{ 6 \\ 0.881 } & 18\\
         \midrule
          {$N_b=5$} & \makecell{ {\#} \\  {Time (s)}} &\makecell{129 \\ \textbf{112.0}} &
         \makecell{ 330/105 \\ \textbf{72.0}/7.93 } &
         \makecell{ 32/105 \\ \textbf{15.2}/5.66 } &
         \makecell{ 99 \\ 0.028 }  &
         \makecell{ 99 \\ 0.836 } &
         \makecell{ 26 \\ 3.37 } & 4\\
         % \midrule
         %  {T-ROM-20} & \makecell{ {\#} \\  {Time (s)}} &\makecell{-- \\ --} &
         % \makecell{ --/-- \\ --/-- } &
         % \makecell{ --/-- \\ --/-- } &
         % \makecell{ -- \\ -- }  &
         % \makecell{ -- \\ -- } &
         % \makecell{ -- \\ -- } & -- & -- \\
         % \midrule
         %  {T-ROM-5} & \makecell{ {\#} \\  {Time (s)}} &\makecell{-- \\ --} &
         % \makecell{ --/-- \\ --/-- } &
         % \makecell{ --/-- \\ --/-- } &
         % \makecell{ -- \\ -- }  &
         % \makecell{ -- \\ -- } &
         % \makecell{ -- \\ -- } & -- & -- \\
         \bottomrule
    \end{tabular}}
    \label{tab:perf2cpu}
\end{table}

\begin{table}[ht]
    \centering
    \caption{Performance metrics for the fixed point iterations using FOM and ROM for $N_b=5$ and $N_b=20$, with $\varepsilon_\mathrm{rb}=10^{-7}$, when using 16 CPUs.}
    {
    \begin{tabular}{@{}lccccccccc@{}}
        \toprule
         {Method} &  {Metric} &  {Iter.} &  {U/T-Assem.} &  {U/T-FOM} &  {ROM} &  {Proj.} &  {SVD} &  {$P$}\\
        \midrule
          {Ref.} & \makecell{ {\#} \\  {Time (s)}} &
          \makecell{104 \\ 21.2} &
          \makecell{104/104 \\ 4.22/0.952}  &
          \makecell{104/104 \\ 11.2/1.77} &
          --  &
         -- &
         -- &
         --\\
         \midrule
          {$N_b=20$} & \makecell{ {\#} \\  {Time (s)}} &\makecell{109 \\ \textbf{17.1}} &
         \makecell{ 280/105 \\ \textbf{7.29}/0.964 } &
         \makecell{ 27/105 \\ \textbf{3.71}/1.79 } &
         \makecell{ 84 \\ 0.014 }  &
         \makecell{ 84 \\ 0.299 } &
         \makecell{ 6 \\ 0.162 } & 18\\
         \midrule
          {$N_b=5$} & \makecell{ {\#} \\  {Time (s)}} &\makecell{129 \\ \textbf{19.3}} &
         \makecell{ 330/105 \\ \textbf{8.64}/0.953 } &
         \makecell{ 32/105 \\ \textbf{4.25}/1.80 } &
         \makecell{ 99 \\ 0.012 }  &
         \makecell{ 99 \\ 0.110 } &
         \makecell{ 26 \\ 0.503 } & 4 \\
         % \midrule
         %  {T-ROM-20} & \makecell{ {\#} \\  {Time (s)}} &\makecell{-- \\ --} &
         % \makecell{ --/-- \\ --/-- } &
         % \makecell{ --/-- \\ --/-- } &
         % \makecell{ -- \\ -- }  &
         % \makecell{ -- \\ -- } &
         % \makecell{ -- \\ -- } & -- & -- \\
         % \midrule
         %  {T-ROM-5} & \makecell{ {\#} \\  {Time (s)}} &\makecell{-- \\ --} &
         % \makecell{ --/-- \\ --/-- } &
         % \makecell{ --/-- \\ --/-- } &
         % \makecell{ -- \\ -- }  &
         % \makecell{ -- \\ -- } &
         % \makecell{ -- \\ -- } & -- & -- \\
         \bottomrule
    \end{tabular}}
    \label{tab:perf16cpu}
\end{table}

\subsection{Comparison of the Error Estimator with the Residual}
\label{section:ROMQM}

The methodology presented in~\Cref{section:methodo} introduces a way to characterize the quality of the ROM at the current iterate, the error estimator $\eta^k$ \eqref{eq:etak}. Other methods can be used as proxies to assess the precision of the ROM, for instance, the residual scaled by the norm of the inverse operator \eqref{eq:errgeneralROM}, or simply the norm of the residual. However, this criterion approximates the error for a single unknown at a given iteration, but does not take into account how much the error propagates from the previous iterations, nor how the error propagates to other unknowns. As seen in the expression of $\eta^k$, if $L$ is close to 1, a non-negligible error component might be carried over from previous iterations. To make the contraction constant larger, a multiplicative factor of $0.95$ is applied to the viscosity $\nu$, which makes the problem stiffer, increasing the convergence rate from $L=0.8825$ to around $L=0.94$ (measured empirically).

When $N_b$ is smaller, a good quality indicator is crucial to detect automatically when to refine the ROM, because when the data are sparse, the error can become larger. To compare the robustness of our estimator $\eta^k$ based on error propagation and the scaled residual \eqref{eq:errgeneralROM}, an experiment is run using $N_b=5$ snapshots and $\varepsilon_\mathrm{rb}=10^{-7}$. We observe what happens to \Cref{algo:acceleratedscheme} when the outer validation loop is not activated (i.e., when convergence with respect to the inexact iterations is observed). Results are shown in \Cref{tab:criterion}. First of all, the pure residual behaves very poorly, with convergence detected despite true errors being orders of magnitude above the required threshold. We observe that for the residual and scaled residual, the true error with respect to the high-fidelity solution is always greater than the set tolerance, which indicates that the convergence was premature. The value of the estimator is always at least one order of magnitude too small, meaning it doesn't give an accurate approximation or upper bound on the actual error. In the case of the propagation-based estimator $\eta^k$, the true error is always below the desired threshold when convergence is attained. Moreover, the value of the estimator itself is a more accurate estimate of the true error. This indicates that measuring the propagation of the ROM error through the iterations and the other unknowns of the system gives a more meaningful and conservative error estimator throughout the iterative process. However, this comes at the cost of a few more iterations and FOM calls.

\begin{table}[ht]
    \caption{Comparison of the residual, scaled residual \eqref{eq:errgeneralROM} and our propagation-based criterion $\eta^k$ as error estimators for the {U-ROM}, with $N_b = 5$ and $\varepsilon_\mathrm{rb}=10^{-7}$.}%
    \centering
    {%\normalsize
    \begin{tabular}{ccccccc}
        \toprule
        Estimator & Abs. Tol. & \#Iter & \#FOM iter. & True Err. & Err. Estim.\\
        \midrule
        Residual &\makecell{ $10^{-2}$\\ $10^{-3}$ \\ $10^{-4}$} & \makecell{ 107 \\ 137 \\ 181} & \makecell{ 6 \\ 6 \\ 11} & \makecell{ 3.08 \\ 3.08 \\ 0.131} & \makecell{ 4.92e-4\\ 4.92e-4 \\ 8.59e-5}\\
        \midrule
        Scaled Residual &\makecell{ $10^{-2}$\\ $10^{-3}$ \\ $10^{-4}$} & \makecell{ 131 \\ 177 \\ 218} & \makecell{ 18 \\ 28 \\ 36} & \makecell{ 1.3e-2 \\ 2.9e-3 \\ 1.85e-4} & \makecell{ 4e-3\\ 4e-4 \\ 7.39e-5}\\
        \midrule
        Propagation & \makecell{ $10^{-2}$\\ $10^{-3}$ \\ $10^{-4}$} & \makecell{ 143 \\ 185 \\ 227} &\makecell{ 29 \\ 37 \\ 45} &\makecell{ 8.72e-3 \\ 8.86e-4 \\ 9.10e-5} & \makecell{ 4.02e-3 \\  9.79e-4 \\ 8.29e-5}\\
         \bottomrule
    \end{tabular}}
    \label{tab:criterion}
\end{table}

\subsection{Speedups for the 3D Extruded Cavity}
\label{sec:3DCavity}

The 2D geometry presented in \Cref{fig:geometry} is extruded along the $z$-axis to make a unit cube. The boundary conditions are kept the same on all extruded faces, with the additional condition $\mathbf{u}_z=0$ on the top face. The problem becomes a 3D heated lid-driven cavity. The tetrahedral mesh is taken to have 50 geometric nodes on each edge, giving a total of 2\,425\,170 DoFs for the velocity and pressure combined, and 775\,607 DoFs for the temperature. When the viscosity \eqref{eq:visco} is multiplied by 5, the problem converges in 27 fixed point iterations to the tolerance $\varepsilon = 10^{-4}$. The results are obtained first using 2$\times$32 CPUs, and then 2$\times$64 CPUs, split across two computing nodes, for U-ROM with $N_b = 5$ and $N_b=10$. Solutions and pointwise errors for $N_b=5$ are shown in \Cref{fig:solutionU3d,fig:solutionT3d}. The maximal errors are of the magnitude of $10^{-9}$.

\begin{table}[ht]
    \centering
    \caption{Performance metrics for the fixed point iterations using FOM and ROM for $N_b=5$ and $N_b=10$, with $\varepsilon_\mathrm{rb}=10^{-7}$, when using 128 CPUs.}
    {
    \begin{tabular}{@{}lccccccccc@{}}
        \toprule
         {Method} &  {Metric} &  {Iter.} &  {U/T-Assem.} &  {U/T-FOM} &  {ROM} &  {Proj.} &  {SVD} &  {$P$}\\
        \midrule
          {Ref.} & \makecell{ {\#} \\  {Time (s)}} &
          \makecell{27 \\ 1133.3} &
          \makecell{27/27 \\ 5.33/0.867}  &
          \makecell{27/27 \\ 1035.9/78.8} &
          --  &
         -- &
         -- &
         --\\
         \midrule
          {$N_b=10$} & \makecell{ {\#} \\  {Time (s)}} &\makecell{28 \\ \textbf{678.1}} &
         \makecell{ 67/28 \\ \textbf{6.92}/0.746 } &
         \makecell{ 15/28 \\ \textbf{574.0}/82.9 } &
         \makecell{ 17 \\ 0.004 }  &
         \makecell{ 17 \\ 0.170 } &
         \makecell{ 4 \\ 0.665 } & 9 \\
         \midrule
          {$N_b=5$} & \makecell{ {\#} \\  {Time (s)}} &\makecell{36 \\ \textbf{713.4}} &
         \makecell{ 83/28 \\ \textbf{8.20}/0.764 } &
         \makecell{ 16/28 \\ \textbf{606.1}/82.0 } &
         \makecell{ 22 \\ 0.004 }  &
         \makecell{ 22 \\ 0.132 } &
         \makecell{ 10 \\ 0.895 } & 4 \\
         \bottomrule
    \end{tabular}}
    \label{tab:perf3D_128cpu}
\end{table}

\Cref{tab:perf3D_128cpu} shows a breakdown of the runtimes for the main sections of the algorithm. Clearly, the main driving cost is the linear solver, and other sections are almost negligible; this is due to the high scalability of assembly with the increasing number of processors, and explains why the speedup is much better than in the 2D case, where just 16 CPUs solved the problem in approximately 20 seconds.

 \begin{figure}[ht]
     \centering
     \subfloat[Velocity field and streamlines.]{
    \includegraphics[trim={1.2in 1in 0.1in 1.2in},clip,width=0.45\linewidth]{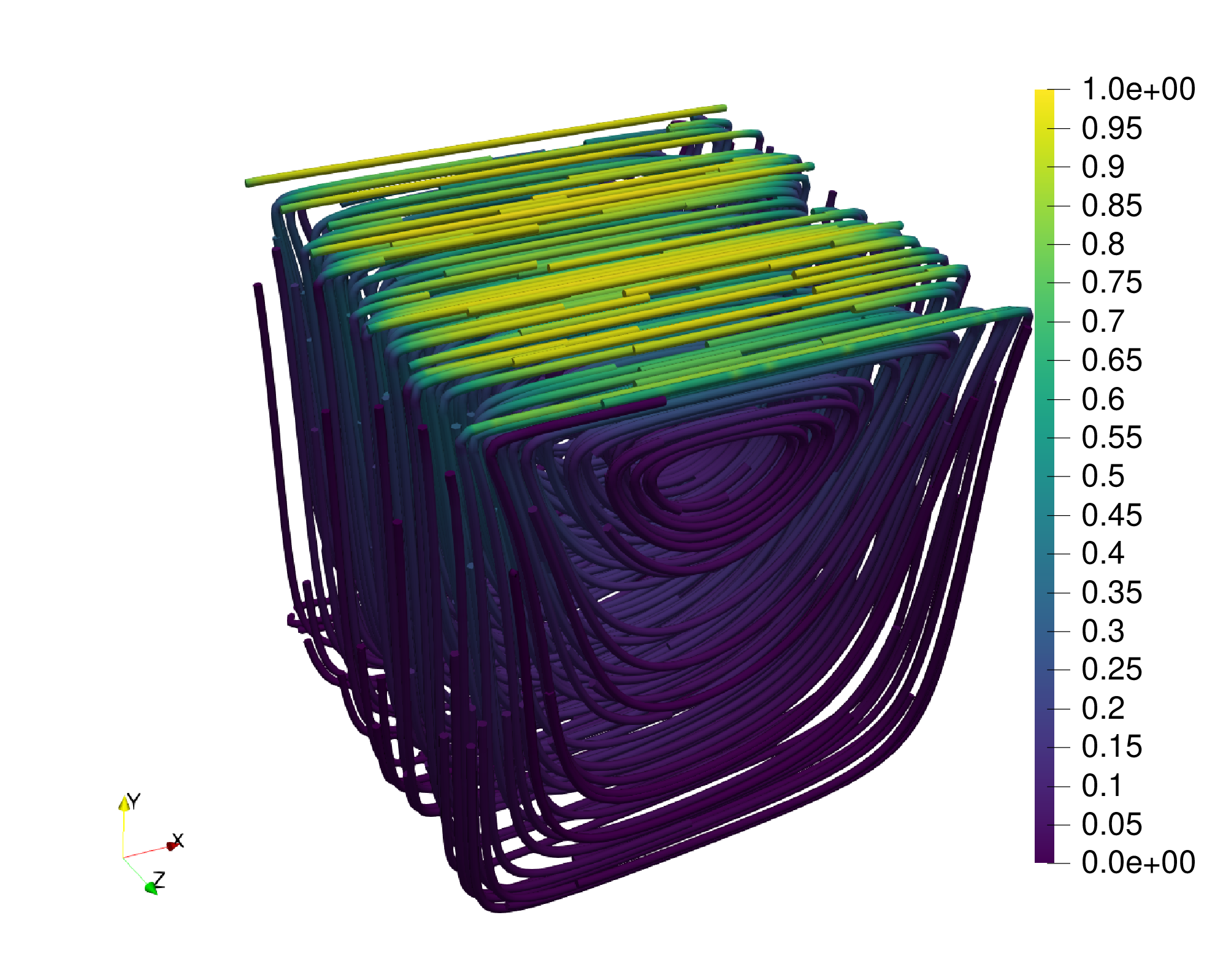}
    }%
     \subfloat[Pointwise error magnitude between the FOM and ROM final velocity field.]{
    \includegraphics[trim={1.2in 1in 0.1in 1.2in},clip,width=0.45\linewidth]{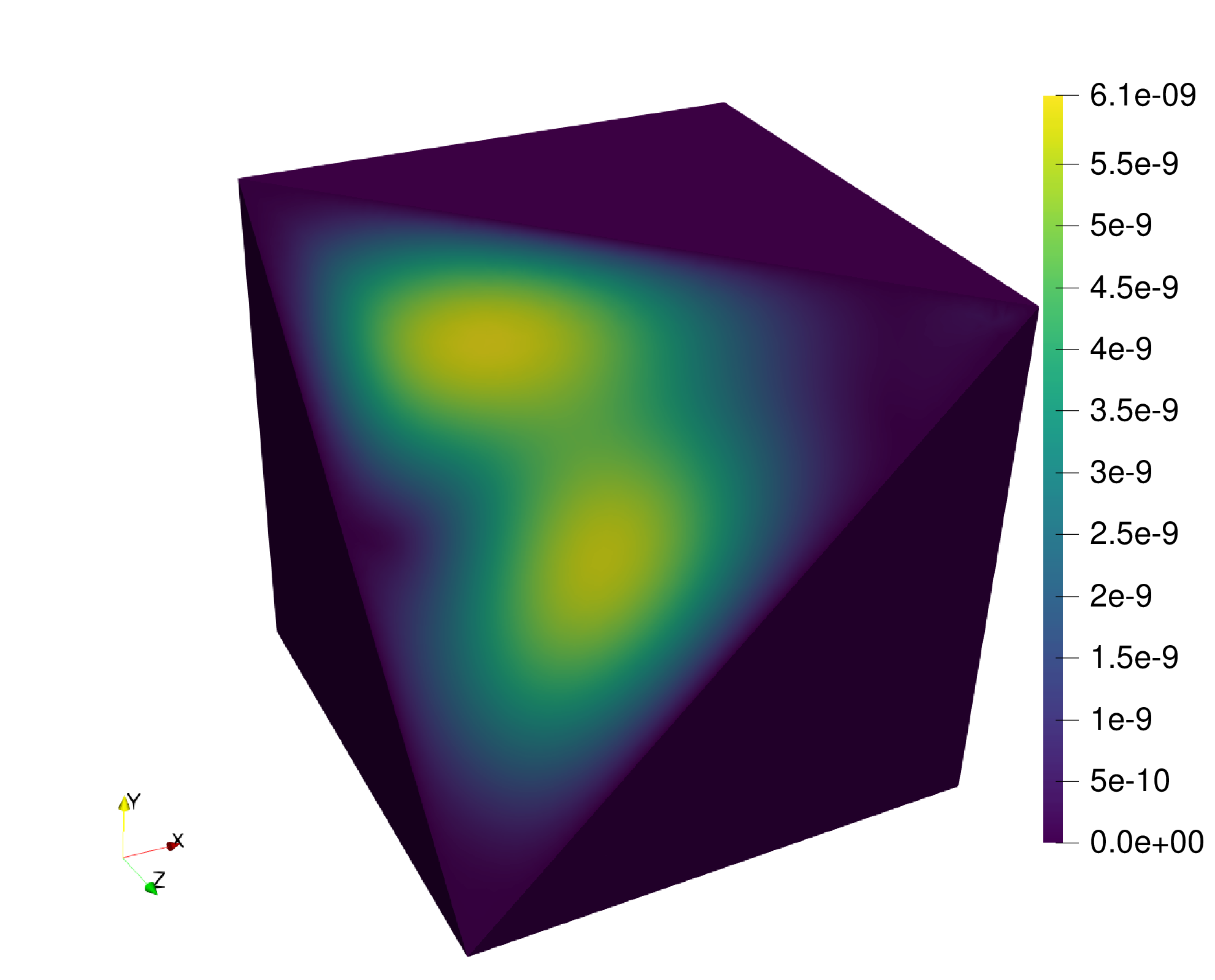}
     }
     \caption{Velocity field and errors w.r.t.\ the reference solution for the solution obtained with the U-ROM-accelerated fixed point scheme with $N_b=5$, $\varepsilon_\mathrm{rb}=10^{-7}$.}%
     \label{fig:solutionU3d}
 \end{figure}

 \begin{figure}[H]
     \centering
     \subfloat[Temperature field and isotherms.]{
    \includegraphics[trim={1.2in 1in 0.1in 1.2in},clip,width=0.45\linewidth]{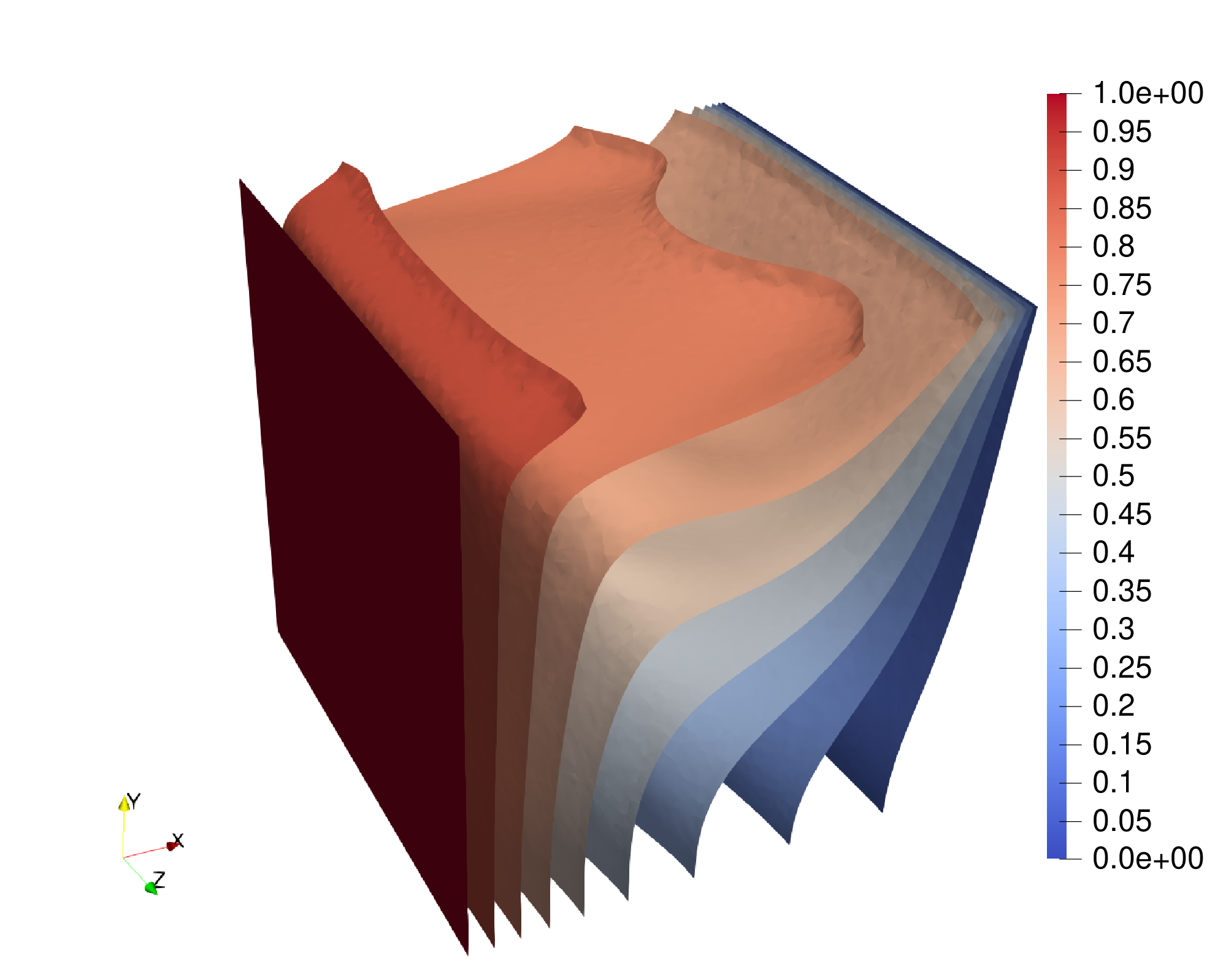}
    }%
     \subfloat[Pointwise error magnitude between the FOM and ROM final velocity field.]{
    \includegraphics[trim={1.2in 1in 0.1in 1.2in},clip,width=0.45\linewidth]{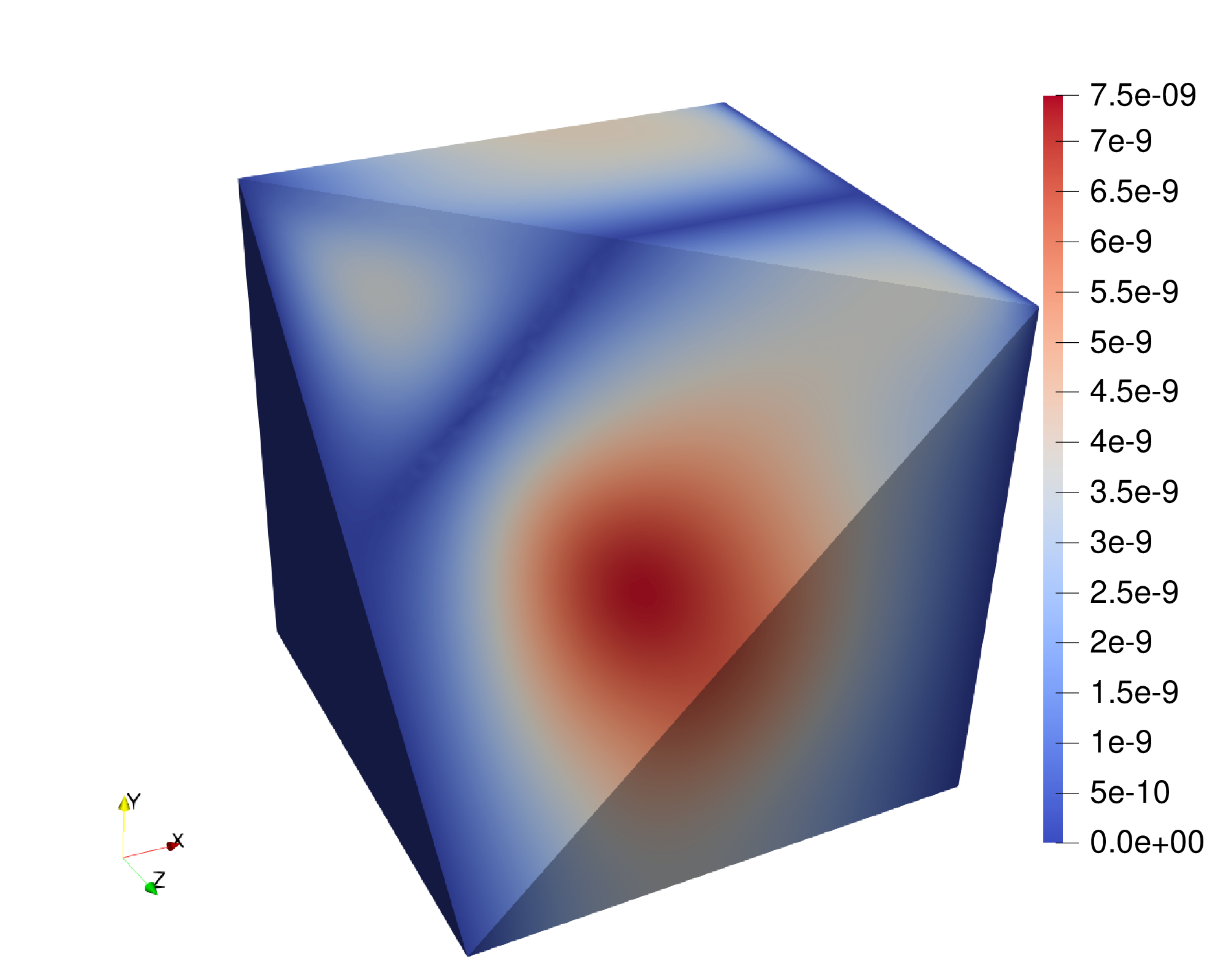}
     }
     \caption{Temperature field and errors w.r.t.\ the reference solution for the solution obtained with the U-ROM-accelerated fixed point scheme with $N_b=5$, $\varepsilon_\mathrm{rb}=10^{-7}$.}%
     \label{fig:solutionT3d}
 \end{figure}

 \begin{figure}[H]
    \centering
    \includegraphics[width=0.45\linewidth]{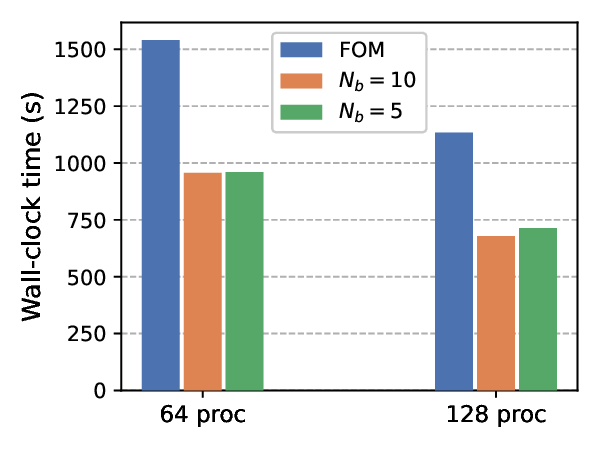}
    \qquad
    \includegraphics[width=0.45\linewidth]{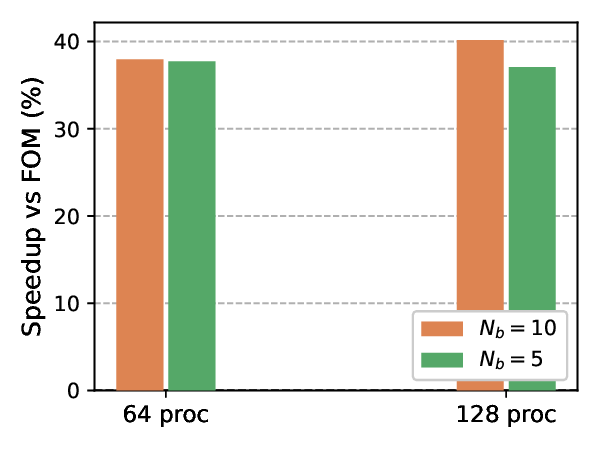}
    \caption{Wall-clock times (left) and corresponding average speedups (right) for each variant of the algorithm for the 3D heated lid-driven cavity.}%
    \label{fig:speedupsbarplot3d}
\end{figure}

The performance statistics are shown in \Cref{fig:speedupsbarplot3d}. It can be seen that $N_b=5$ is a bit slower, because it requires 5 more iterations in total, and one additional FOM solve. Globally, the observed speedups are between 35\% and 40\%, because the main driving cost of the algorithm for the problem size is the linear solver.

\section{Conclusion}

This article presents the formalism for a general accelerated inexact fixed point algorithm with an on-the-fly quality measurement based on error propagation. This method is combined with PROMs to accelerate on-the-fly the solution to PDE-adjacent problems. General convergence results for fixed point iterations depending on the solution of multiple auxiliary equations are given. Error estimates are given in the case where those auxiliary equations are solved inexactly using a PROM, to integrate into the proposed algorithm. Results on a multiphysics nonlinear problem demonstrate that significant speedups can be attained (up to 20\% in the 2D case and 40\% in the 3D case) without sacrificing precision and at no additional offline cost. For the incompressible Navier-Stokes equations at low Reynolds number, the algorithm is stable to changes in hyperparameters like the size of the reduced basis.

There are multiple possible follow-ups to this work. First of all, since the methodology here was inspired by on-the-fly ROM applied to optimization, it would be interesting to apply directly the methodology to a convex optimization problem, or to generalize~\Cref{cor:main} to more general cases of nonlinear and nonconvex optimization (where the contraction principle does not necessarily apply). Moreover, this methodology could be applied to solve unsteady problems. For instance, an accelerated fixed point could be used at each timestep to solve the equations. Because of warm starting (the initial guess for the fixed point is the final solution at the previous timestep), fewer fixed point iterations might be needed to reach convergence, thus potentially limiting the effectiveness of the proposed method. However, one could hope that reusing snapshots from previous time steps might not introduce too much error in the process, and the sampling phase might be skipped altogether, which could compensate for the aforementioned shortcoming. More theoretical analysis should be done to consider the case when the ROM error is propagated through timesteps. Another potential avenue is to reduce the total number of large-scale matrix assemblies, since the assembly phase is the dominant cost in the ROM algorithm, especially for smaller 2D problems; tensor-valued regression or empirical interpolation might be a possible solution to this problem. This would allow us to remove \Cref{Hypo:smallcost} altogether and allow for an efficient application of the algorithm to a broader class of problems. Fixed point methods that are already accelerated by other means, like projection-based coupling schemes~\citep{Deteix_Jendoubi_Yakoubi_2014} or Aitken acceleration~\citep{Irons_Tuck_1969,Vierendeels_Lanoye_Degroote_Verdonck_2007,Ramière_Helfer_2015}, could be further improved using these ideas. There is also theoretical interest in situating this methodology in a larger class of numerical algorithms like Anderson acceleration and Reduced-Rank Extrapolation~\citep{Fang_Saad_2009,Sidi_2016,Pollock_Rebholz_2025} or multipoint methods~\citep{Traub_1964,Petković_Neta_Petković_Džunić_2014}.

\section*{Statements and Declarations}

\subsection*{Declaration of Competing Interest}

The authors declare that they have no competing interests (financial or personal) that could have appeared to influence the research presented in this paper.

\bibliography{references2}

\end{document}